\documentclass[12pt]{amsart}
\usepackage{amsfonts, amsmath, amsthm}

\usepackage{graphicx}
\usepackage{epstopdf}
\usepackage{psfrag}
\usepackage{color}

\newtheorem{pro}{Proposition}[section]
\newtheorem{thm}[pro]{Theorem}
\newtheorem{lem}[pro]{Lemma}
\newtheorem{clm}[pro]{Claim}
\newtheorem{prob}[pro]{Problem}

\newtheorem{cor}[pro]{Corollary}
\newtheorem{quest}[pro]{Question}

\theoremstyle{definition}
\newtheorem{dfn}[pro]{Definition}
\newtheorem{ex}[pro]{Example}
\newtheorem{rmk}[pro]{Remark}

\theoremstyle{remark}

\title{Computing Heegaard Genus is NP-Hard}
\date{\today}
\address{Pitzer College}
\email{bachman@pitzer.edu}
\author{David Bachman}
\thanks{The authors would like to thank Martin Tancer, whose guidance was instrumental at the genesis of this project, and Iyad Kanj and Marcus Schaefer for helpful conversations along the way.}
\begin{document}

\address{Quest University}
\email{rdt@questu.ca}
\author{Ryan Derby-Talbot}

\address{DePaul University}
\email{esedgwick@cdm.depaul.edu}
\author{Eric Sedgwick}

\dedicatory{Dedicated to the memory of Ji\v{r}\'{\i} Matou\v{s}ek }

\newcommand{\NP}{{\sc NP}}
\newcommand{\band}{{\sc and}}
\newcommand{\bvar}{{\sc var}}
\newcommand{\bor}{{\sc or}}
\newcommand{\brep}{{\sc rep}}
\newcommand{\bnot}{{\sc not}}
\newcommand{\bcap}{{\sc cap}}
\newcommand{\bend}{{\sc end}}

\begin{abstract}
We show that {\sc Heegaard Genus $\leq g$}, the problem of deciding whether a triangulated 3-manifold admits a Heegaard splitting of genus less than or equal to $g$, is \NP-hard.  The result follows from a quadratic time reduction of the \NP-complete problem {\sc CNF-SAT} to {\sc Heegaard Genus $\leq g$}.


\end{abstract}

\maketitle

\newcommand{\LQ}{\ensuremath{|Q|}}

\section{Introduction}

While there is a tradition of studying decision problems in 3-manifold topology, the historical focus has been showing that problems are decidable \cite{haken-unknot, rubinstein, thompson, johannson1,johannson2,li-heegaard-genus, jaco-sedgwick, matousek-sedgwick-tancer-wagner}.
More recently, the computational complexity of these and related problems has gained attention \cite{hass-lagarias-pippenger,agol-hass-thurston, schleimer, burton-spreer-taut,  burton-verdiere-demesmay, burton-demesmay-wagner, lackenby-sublink}.
Here we show that one of the most basic decision problems for 3-manifolds, the problem of determining Heegaard genus, is \NP-hard.

Every closed, orientable 3-manifold $M$ has a {\em Heegaard surface}: a closed surface that splits the manifold into a pair of handlebodies (i.e., thickened graphs). The Heegaard genus, $g(M)$, is the minimal genus of a Heegaard surface for $M$, and is one of the most basic 3-manifold invariants.  Because Heegaard surfaces are generic, they have been studied extensively and have been effectively classified for large classes of manifolds \cite{moriah-schultens, kobayashi-2-bridge}.  It is thus natural to ask (phrased as a decision problem):

\begin{prob}{\sc Heegaard Genus $\leq g$:}
Given a triangulated 3-manifold $M$ and a natural number $g$, does $M$ have a Heegaard surface of genus $\leq g$?
\end{prob}

{\sc Heegaard Genus $\leq g$} was shown to be decidable (computable) by Johannson \cite{johannson1,johannson2} in the Haken case and by Li in the non-Haken case \cite{li-heegaard-genus}.  Our main result is the following:

\begin{thm}
\label{t-genus-hard}
{\sc Heegaard Genus $\leq g$} is \NP-hard.
\end{thm}


One way of obtaining a Heegaard surface in certain 3-manifolds is to {\it amalgamate} Heegaard surfaces in submanifolds. This approach allows us to relate Heegaard genus to satisfiability of Boolean formulas in {\em conjunctive normal form}, that is Boolean formulas stated as a conjunction of disjunctions, for example:
\[Q=(a \lor c) \land (\lnot a \lor b) \land (b \lor c)\]

We will let \LQ~ denote the length of $Q$ without counting parentheses, e.g.~\LQ = 12 for the above example.

\begin{prob}{\sc CNF-SAT:}
Given $Q$, a Boolean formula in conjunctive normal form, is there a satisfying assignment (i.e., an assignment of truth values to the variables) that makes the formula true?
\end{prob}

{\sc CNF-SAT} is well known to be \NP-complete.  We prove Theorem \ref{t-genus-hard} by giving a polynomial (quadratic) time reduction of {\sc CNF-SAT} to {\sc Heegaard Genus $\leq g$}.  Our reduction will proceed in two steps, first proving that there are manifolds $M_Q$ that encode a formula $Q$:

\theoremstyle{theorem}
\newtheorem*{proA}{Proposition \ref{p-q-mq}}

\begin{proA}
\label{prop1}
Let $Q$ be an instance of {\sc CNF-SAT}.  Then there is a manifold $M_Q$ with Heegaard genus $g(M_Q) \geq \LQ + 2$, with equality holding if and only if $Q$ has a satisfying assignment. \end{proA}

\begin{figure}
\psfrag{a}{$a$}
\psfrag{b}{$b$}
\psfrag{c}{$c$}
\psfrag{A}{$\lnot a$}
\psfrag{B}{$a \lor c$}
\psfrag{C}{$\lnot a \lor b$}
\psfrag{D}{$b \lor c$}
\psfrag{E}{$(a \lor c) \land (\lnot a \lor b)$}
\psfrag{F}{$((a \lor c) \land (\lnot a \lor b)) \land (b \lor c)$}
\includegraphics[width=4in]{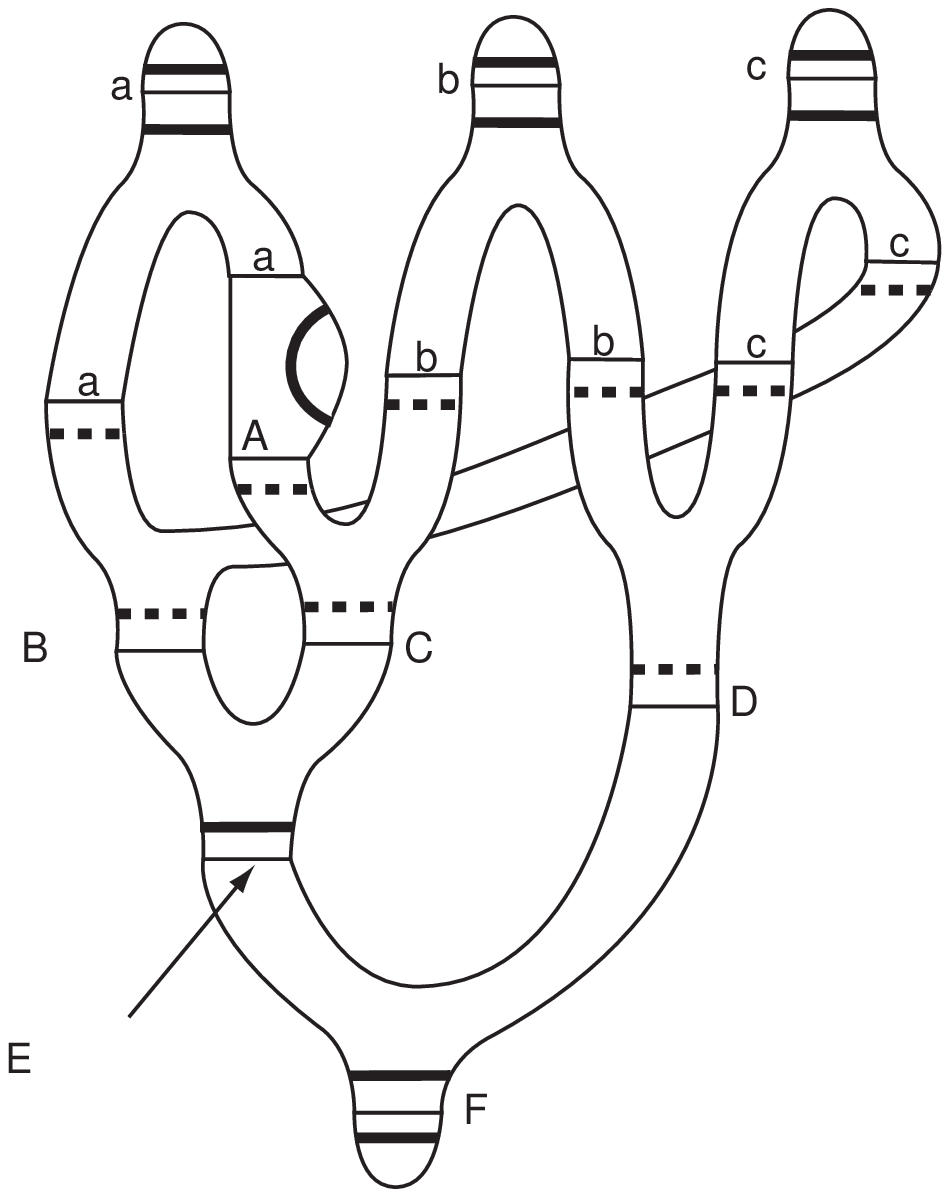}
\caption{The construction of $M_Q$, where $Q=((a \lor c) \land (\lnot a \lor b)) \land (b \lor c)$.}
\label{f:MQexample}
\end{figure}

The proof of Proposition~\ref{p-q-mq} is based on constructing $M_Q$ as a direct translation of the formula $Q$ (a schematic of $M_Q$ for the aforementioned $Q$ is shown in Figure \ref{f:MQexample}), formed by taking a collection of Heegaard genus two ``block'' manifolds, one block for each term (\bvar(iable), \brep(licate), \bnot, \band, \bor) in $Q$, and gluing them together along torus boundary components via high distance maps.  Each gluing surface then represents a sub-statement of $Q$.  The high-distance gluings guarantee that any minimal genus Heegaard surface for $M_Q$ is an amalgamation of Heegaard surfaces of the blocks (we provide a proof of this fact in the appendix of this paper), and this allows us to compute the Heegaard genus of $M_Q$.

Every Heegaard surface induces a {\it bipartition}, a partition into two sets, of its manifold's boundary components.  The blocks are constructed so that each block emulates its logical operator via the way its minimal genus Heegaard surfaces bipartition its boundary components.  The {\sc or} block is flexible, in that every non-trivial bipartition is possible, whereas all other block types have a fixed bipartition of boundary components determined by the minimal genus Heegaard surfaces.  When $Q$ is satisfiable, there is a minimal genus Heegaard surface for each block so that the complementary pieces can be bicolored in a particular way (see Definition~\ref{dfn:GHS}) so that the Heegaard surfaces for the blocks can be amalgamated to a genus $\LQ+2$ Heegaard surface for $M_Q$.  The converse uses the same setup.  We show that the genus of $M_Q$ is at least $\LQ+2$, and that when equality is achieved it is possible to read off a satisfying assignment for $Q$ from a bicoloring induced by Heegaard surfaces for the block manifolds.

There are many manifolds that fit the above description of $M_Q$.  The second step, from which Theorem \ref{t-genus-hard} follows, is that we can construct a triangulation for one efficiently.

\theoremstyle{theorem}
\newtheorem*{proB}{Proposition \ref{t-triangulation-quadratic}}

\begin{proB}
A triangulated $M_Q$ can be produced in quadratic time (and tetrahedra) in \LQ.
\end{proB}

   The essential ingredient for our main result is our ability to choose block manifolds whose minimal genus Heegaard surfaces bipartition their boundary components in a way that emulates the required logical operators. It is then worth asking: given a set of bipartitions, is there a 3-manifold whose minimal genus Heegaard surfaces induce precisely that set? In fact, this is an easy corollary of the techniques we use here.

\newtheorem*{qcor}{Corollary \ref{cor-partitions}}

\begin{qcor}
Let $\mathcal P$ be a non-empty set of bipartitions of  $1,2,...,n$.  Then there is a 3-manifold $X$ and a numbering of its boundary components, $1,2,...,n$,  so that the set of bipartitions of $\partial X$ induced by minimal genus Heegaard splittings of $X$ is precisely $\mathcal P$.
\end{qcor}

This paper is organized as follows:  Section \ref{s-hs-amalg} contains the required background on Heegaard splittings, surfaces, and amalgamation. Section \ref{s:constructing} gives a recipe for producing $M_Q$ and proves Proposition~\ref{p-q-mq} and Corollary \ref{cor-partitions}. Section \ref{s:triangulating} shows how to triangulate $M_Q$ and proves Proposition \ref{t-triangulation-quadratic}.   Section \ref{s:questions} lists some related open questions, and Section \ref{s:appendix} is an appendix that proves Proposition \ref{t:amalgamation}, which explains how high distance gluings ensure that minimal genus Heegaard surfaces are amalgamations.


\section{Heegaard splittings and amalgamations}
\label{s-hs-amalg}

\begin{dfn}
Consider a 3-ball $B$, and attach 1-handles to $\partial B$. The resulting 3-manifold is a {\em handlebody}. Alternatively, let $F$ be a closed, not necessarily connected, orientable surface such that each component of $F$ has genus greater than zero. Take the product $F \times [0,1]$ and attach 1-handles along $F \times \{ 1 \}$. Assuming it is connected, the resulting 3-manifold $V$ is a {\em compression body}, and we denote $\partial_- V = F \times \{0\}$ and $\partial_+ V = \partial V - \partial_- V$. (We will consider a handlebody as a compression body with $\partial_-V = \emptyset$.)
\end{dfn}

Let $M$ denote a compact, connected, orientable 3-manifold.

\begin{dfn}
A {\em Heegaard splitting} for $M$ is a decomposition $M = V \cup W$ where $V$ and $W$ are compression bodies such that $\partial_+ V = \partial_+ W = V \cap W$. The surface $H = \partial_+ V = \partial_+ W$ in $M$ is called a {\em Heegaard surface}, and when needed we may include this surface in the notation for the Heegaard splitting as $V \cup_H W$. The {\em genus} of $V \cup_H W$ is the genus of $H$, denoted $g(H)$.
\end{dfn}

\begin{rmk}
Note that the compression bodies $V$ and $W$ bipartition the boundary of $M$ into $\partial_V M = \partial M \cap V = \partial_-V$ and $\partial_W M = \partial M \cap W = \partial_- W$. In particular, a Heegaard splitting for $M$ always induces a bipartition $\{ \partial_V M | \partial_W M\}$ of the boundary components of $M$, and thus it is proper to say that $V \cup W$ is a Heegaard splitting of $M$ with respect to the bipartition $\{ \partial_V M | \partial_W M\}$.
\end{rmk}

Given $M$, one can find Heegaard splittings of $M$ in several ways. For example, if $M$ is triangulated with $t$ tetrahedra, then one can obtain a Heegaard splitting of $M$ of genus $t+1$, taking the boundary of a regular neighborhood of the 1-skeleton as the Heegaard surface. Alternatively, if $M$ can be decomposed as a union of submanifolds $M = \bigcup M_i$, so that $M$ is obtained by gluing the $M_i$ together along their boundary components (including possible self-gluings), one can potentially {\em amalgamate} Heegaard splittings of the $M_i$ to form a Heegaard splitting of $M$:

\begin{ex}
\label{ex:amalgamation}
Let $M_1$ and $M_2$ be 3-manifolds such that $\partial M_1 \cong \partial M_2 \cong F$, and let $V_1 \cup W_1$ be a Heegaard splitting of $M_1$ with respect to the bipartition $\{ \emptyset | \partial M_1\}$ and $V_2 \cup W_2$  a Heegaard splitting of $M_2$ with respect to the bipartition $\{ \partial M_2 | \emptyset\}$. Note that both $W_1$ and $V_2$ are compression bodies of the form $F \times [0,1] \cup \{\mbox{1-handles}\}$. Form the 3-manifold $M$ by gluing $M_1$ to $M_2$ along their boundaries, and, abusing notation slightly, let $F$ be the image of the boundary components in $M$. Collapse the product structures in $W_1$ and $V_2$ so that in each, $F \times [0,1]$ is mapped to $F \times \{ 0 \} = F$, and so that the 1-handles of each of $W_1$ and $V_2$ are attached disjointly on $F$. We then obtain a new Heegaard splitting $V \cup W$ of $M$, where $V = V_1 \cup \{\mbox{1-handles in $V_2$}\}$, and $W = \{\mbox{1-handles in $W_1$}\} \cup W_2$. The splitting $V \cup W$ is called the {\em amalgamation} of $V_1 \cup W_1$ and $V_2 \cup W_2$ along $F$. See Figure~\ref{fig:amalgamation}.
\end{ex}

\begin{figure}
\psfrag{p}{$M_1$}
\psfrag{q}{$M_2$}
\psfrag{M}{$M$}
\psfrag{a}{$V_1$}
\psfrag{c}{$V_2$}
\psfrag{b}{$W_1$}
\psfrag{d}{$W_2$}
\psfrag{V}{$V$}
\psfrag{W}{$W$}
\psfrag{f}{$\partial M_1$}
\psfrag{h}{1-handles}
\psfrag{g}{$F$}
   \centering
   \includegraphics[width=5in]{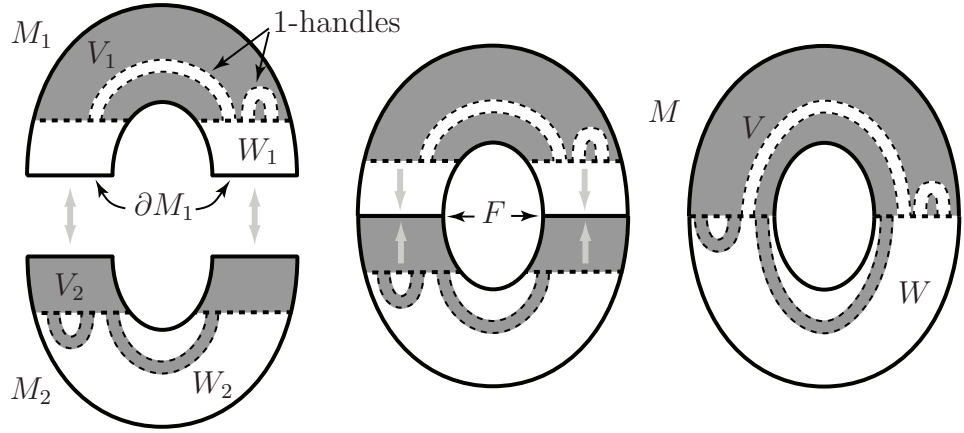}
   \caption{A schematic for the amalgamation given in Example~\ref{ex:amalgamation}. The light and dark regions represent compression bodies, with $W_1$ and $V_2$ expressed as $F \times [0,1] \cup (\mbox{1-handles})$. The dotted lines represent Heegaard surfaces.}
   \label{fig:amalgamation}
\end{figure}

Constructing an amalgamation of $M = \bigcup M_i$ from component Heegaard splittings of $M_i$, however, is not always possible.

\begin{ex}
\label{ex:non-amalgamation1}
Suppose $M$ is formed by taking $M_1 = T^2 \times [0,1]$ and gluing the two components of $\partial M_1$ together. Let $F$ be the image of $\partial M_1$ (an embedded torus) in $M$.

It is well known that $M_1$ admits two irreducible Heegaard surfaces up to isotopy \cite{Scharlemann-Thompson}:  a ``Type 1'' surface that is a level torus $T^2 \times \{\frac 1 2\}$ and induces the non-trivial bipartition of boundary components $\{ T^2 \times \{0\} |T^2 \times \{1\} \}$, and a ``Type 2'' surface that is a genus two Heegaard surface obtained by tubing together two disjoint copies, say $T^2 \times \{\frac 1 4\}$ and $T^2 \times \{\frac 3 4\}$,  of the level surface. Note that this latter surface induces the trivial bipartition of boundary components $\{ T^2 \times \{0\}, T^2 \times \{1\} | \emptyset\}$.

One {\em cannot} form an amalgamated splitting for $M$ by taking a Type 2 Heegaard splitting of $M_1$ and amalgamating it to itself. (See Figure~\ref{fig:non-amalgamations1}(a).) This is because in attempting to apply the construction of Example~\ref{ex:amalgamation}, we do not end up with two resulting compression bodies once we collapse the product structure of $F \times [0,1]$ (i.e.~the resulting ``Heegaard surface'' is not separating).
\end{ex}

\begin{ex}
\label{ex:non-amalgamation2}
Let $M_1$ and $M_2$ each be copies of $T^2 \times [0,1]$, and form $M = M_1 \cup M_2$ by gluing $\partial M_1$ to $\partial M_2$ component-wise. Let $F = \partial M_1 = \partial M_2$, so that $F$ consists of two disjoint tori embedded in $M$. Then, one cannot form an amalgamated Heegaard splitting of $M$ from Type 1 Heegaard splittings of $M_1$ and $M_2$. (See Figure~\ref{fig:non-amalgamations1}(b).) The issue here is that the Heegaard splitting of $M_i$, $i = 1,2$, does not partition the components of $\partial M_i$ into a single compression body, and thus one cannot simultaneously collapse the product structure $F \times [0,1]$ along each component of $F$ as in Example~\ref{ex:amalgamation} to form an amalgamation.
\end{ex}

\begin{figure}
\psfrag{a}{(a)}
\psfrag{b}{(b)}
\psfrag{g}{$\mathcal G$}
   \centering
   \includegraphics[width=3.5in]{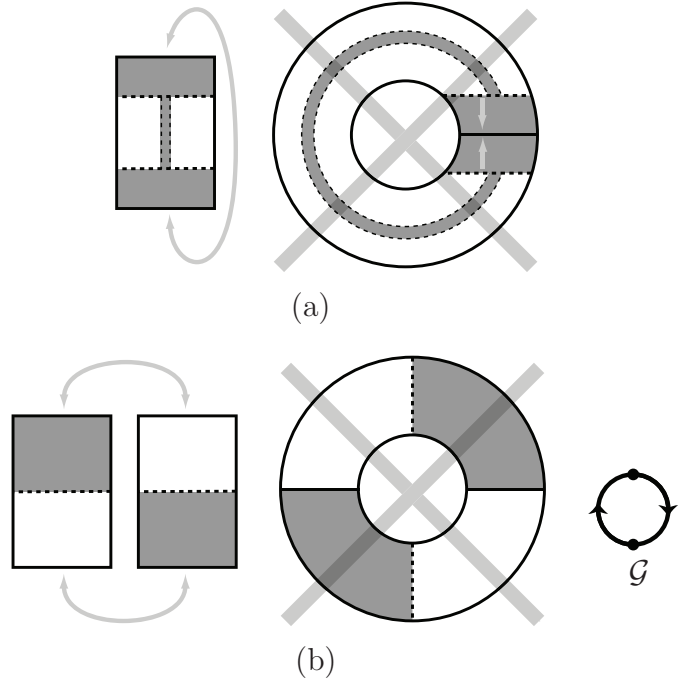}
   \caption{(a) A Type 2 Heegaard splitting of $T^2 \times [0,1]$ cannot be amalgamated to itself; (b) two Type 1 Heegaard splittings of $T^2 \times [0,1]$ cannot be amalgamated together (note that $\mathcal G$ here is not a DAG).}
   \label{fig:non-amalgamations1}
\end{figure}

Assume that $M = \bigcup M_i$ where the $M_i$ meet along boundary components. Rather than thinking of the $M_i$ in a linear order, it is more natural to consider the following construction. Let $\mathcal G$ be the dual graph of $\bigcup M_i$, so that each submanifold $M_i$ is assigned a vertex $x$, and two vertices corresponding to $M_i$ and $M_j$ are connected by an edge for each component of $\partial M_i \cap \partial M_j$. (Note that $i$ may equal $j$, in the case of self-gluings.) Relabelling the submanifolds $M_i$ as $M_x$, one for each vertex $x$ of $\mathcal G$, we can consider $M = \bigcup_{x \in \mathcal G} M_x$. The following definition provides the conditions under which Heegaard splittings of the $M_x$ can form an amalgamated Heegaard splitting of $M$.

\begin{dfn}
\label{dfn:GHS}
A {\it generalized Heegaard splitting} of $M = \bigcup_{x \in \mathcal G} M_x$ is a choice, for each $M_x$, of a Heegaard splitting $M_x = V_x \cup W_x$, so that:

\begin{enumerate}
\item The compression bodies are bicolored ``black'' and ``white'' (or ``V'' and ``W''). That is $V_x \cap V_{x'} = \emptyset, W_x \cap W_{x'} = \emptyset$, for all $x \neq x'$.
\item Given this bicoloring, the graph $\mathcal G$ becomes a directed acyclic graph (DAG) after assigning edges of $\mathcal G$ to point toward ``white'': as each edge $e$ of $\mathcal G$ is dual to a surface in $M$ that has a black compression body $V_x$ on one side and a white compression body $W_{x'}$ on the other, assign an orientation to $e$ that points from $x$ to $x'$ (``black'' to ``white''). We require that the resulting directed graph has no directed cycles. 	
\end{enumerate}
\end{dfn}

\begin{thm}
\label{thm:GHS}
If $\bigcup_{x \in \mathcal G} \left(V_x \cup W_x\right)$ is a generalized Heegaard splitting of $M = \bigcup_{x \in \mathcal G} M_x$, then the Heegaard splittings $V_x \cup W_x$ can be amalgamated to form a Heegaard splitting of $M$.
\end{thm}

\begin{proof}
We construct the desired Heegaard splitting in stages. Assume that the graph $\mathcal G$ is directed as per Definition~\ref{dfn:GHS}. As $\mathcal G$ contains no directed cycles, the graph has a vertex which is a sink (all edges meeting it point ``in''). Remove this vertex and all edges meeting it from the graph. In the remaining (potentially disconnected) graph, find another sink, and repeat the process. Continue until all such sinks have been removed. As $\mathcal G$ is a DAG, this means we are left only with a collection of vertices (the sources of the original graph).

Now add back the last removed sink $x_0$, along with the edges $e_1, \ldots, e_m$ that point in toward it. Let $x_1, \ldots, x_n$ be the set of vertices that bound the edges $e_1, \ldots, e_m$ along with $x_0$. Since $x_0$ is a sink, the bicoloring of the compression bodies of $\bigcup V_{x_i} \cup W_{x_i}$ in the generalized Heegaard splitting implies $M_{x_0}$ meets each $M_{x_i}$ only in $W_{x_0}$ and $V_{x_i}$, $i=1, \ldots, n$, respectively. In particular, the components $F_{e_1}, \ldots, F_{e_m}$ of $\partial M_{x_0}$ corresponding to the edges $e_1, \ldots, e_m$ are all contained in $W_{x_0}$ and $\bigcup V_{x_i}$. Thus, we may carry out the procedure of Example~\ref{ex:amalgamation} and collapse the product structures $F_{e_j} \times [0,1]$ to $F_{e_j}$ simultaneously for all $j$ in the compression bodies $W_{x_0}$, $V_{x_1}, \ldots, V_{x_n}$ and obtain a new Heegaard splitting $V' \cup W'$ of $M' = M_{x_0} \cup \ldots \cup M_{x_n}$. Note that this new Heegaard splitting preserves the original bicoloring given by $\bigcup V_x \cup W_x$ for boundary components of $M'$: if $F'$ is a component of $\partial M'$, then $F' \subset \partial V'$ if and only if $F' \subset \partial V_{x_i}$ for some $x_i$. (Boundary components of $M'$ stay ``black'' or ``white.'')

Add back in the next sink $x_0'$. If $M_{x_0'}$ does not meet $M'$, then we simply repeat the above process for the subset of $\mathcal G$ that  consists of edges and bounding vertices that meet $x_0'$. If $M_{x_0'}$ meets $M'$, then we consider $M'$ as a whole with the Heegaard splitting $V' \cup W'$ obtained above. Since $V' \cup W'$ preserves the bicoloring of boundary components of $M'$ given by the original generalized Heegaard splitting, we can repeat the above process to obtain a new Heegaard splitting of $M_{x_0'} \cup M' \cup \{ M_y \ |\  y \ \mbox{is a new vertex directed towards $x_0'$}\}$.

Building in this way, we can continue to obtain new Heegaard splittings of larger collections of submanifolds of $M$, until we complete the graph $\mathcal G$ and produce a Heegaard splitting $V \cup W$ of $M$.
\end{proof}

As before, the Heegaard splitting $V \cup W$ obtained in the above proof is called the amalgamation of the Heegaard splittings of the $M_x$ along the surfaces $F$, where $F$ is the collection of components of the $\partial M_x$ that are dual to edges in $\mathcal G$ (i.e.~ $F = \left( \bigcup_{x \in \mathcal G} \partial M_x \right) \setminus \partial M$). Note that $V \cup W$ is obtained by sequential applications of the technique in Example~\ref{ex:amalgamation} to amalgamations of Heegaard splittings of ``sink'' submanifolds to their adjacent submanifolds. The critical feature of a generalized Heegaard splitting that allows one to construct $V \cup W$ is that each component Heegaard splitting bipartitions the boundary components of the $M_x$ suitably so that we can bicolor the set of compression bodies (this allows us to end up with two compression bodies in the amalgamated Heegaard splitting, avoiding the problem of Example~\ref{ex:non-amalgamation1}), and can use the bicoloring to direct the edges of $\mathcal G$ so that we can amalgamate in sequence ``outward'' from sinks at each stage (thereby avoiding the problem of Example~\ref{ex:non-amalgamation2} --- recall Figure~\ref{fig:non-amalgamations1}).

\begin{thm}
\label{t:GenusCount}
Suppose $\bigcup_{x \in \mathcal G} \left(V_x \cup_{H_x} W_x\right)$ is a generalized Heegaard splitting of $M = \bigcup_{x \in \mathcal G} M_x$. For every edge $e$ of $\mathcal G$, let $F_e$ denote the component of $\bigcup \partial M_x$ dual to $e$ in $M$. Let $V \cup_H W$ be the amalgamation of $\bigcup_{x \in \mathcal G} \left(V_x \cup_{H_x} W_x\right)$. Then $$g(H) = \displaystyle{\sum_{x \in \mathcal G} g(H_x) - \sum_{e \in \mathcal G} g(F_e) + 1 - \chi(\mathcal G)}.$$
\end{thm}

\begin{proof}
Proceed with the same setup and notation as in the proof of Theorem~\ref{thm:GHS}. In particular, for the first step in constructing an amalgamation of $M$, consider Heegaard splittings $V_{x_i} \cup_{H_{x_i}} W_{x_i}$ of $M_{x_i}$, $i=0, \ldots, n$, respectively, and their corresponding vertices $x_0, \ldots, x_n$ and connecting edges $e_1, \ldots, e_m$ in $\mathcal G$. Let $F_{e_1}, \ldots, F_{e_m}$ denote the corresponding surfaces in $M$ dual to $e_1, \ldots, e_m$. Let $M' = \bigcup_{i = 0}^n M_{x_i}$.

By construction, the genus of the amalgamated Heegaard splitting is obtained by adding the genus of $H_{x_0}$ to the {\em handle numbers} of $V_{x_i}$, $i = 1, \ldots, n$. If $V$ is a compression body, then the handle number of $V$ is the number of 1-handles added to $\partial_-V \times [0,1]$ along $\partial_-V \times \{1\}$ to obtain $V$ (see Figure~\ref{fig:handle_number}). There are two types of potential such 1-handles: a minimal set that connects components of $\partial_- V \times [0,1]$ (essentially fulfilling the role of ``connected sum" of components of $\partial_- V \times \{1\}$), and those that increase the genus of $\partial_+ V$. Thus, the handle number of $V$ equals $$\#_{handle}(V) = g(\partial_+ V) - \sum_{F \in \partial_- V} g(F) + |\partial_- V| - 1.$$

\begin{figure}
   \centering
  \includegraphics[width=3in]{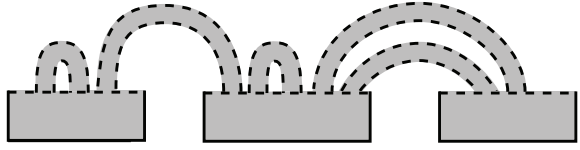}
   \caption{A schematic of a compression body $V$ with $\#_{handle}(V) = 5$. (Note that $\partial_+ V$ is denoted by dotted lines.)}
   \label{fig:handle_number}
\end{figure}

Let $V' \cup_{H'} W'$ be the amalgamation of $\bigcup_{i=0}^n \left( V_{x_i} \cup_{H_{x_i}} W_{x_i} \right)$. Using the handle number, the genus of the Heegaard surface $H'$ is
$$g(H') = g(H_{x_0}) + \sum_{i=1}^n \#_{handle}(V_{x_i}).$$
\noindent Plugging in the equations for the handle numbers for the $V_{x_i}$ produces

$$g(H') = g(H_{x_0}) + \sum_{i=1}^n g(H_{x_i}) - \sum_{j=1}^m g(F_{e_j}) + \left|\bigcup_{j=1}^m F_{e_j}\right| - n$$
$$ = \sum_{i = 0}^n g(H_{x_i}) - \sum_{j=1}^m g(F_{e_j}) + m - n.$$

Let $\mathcal G'$ denote the graph connecting $x_0$ to $x_1, \ldots, x_n$. Since $m$ is the number of edges in $\mathcal G'$ and $n$ is the number of vertices minus one, we conclude $m - n = 1 - \chi(\mathcal G')$. Hence
$$g(H') = \sum_{i=0}^n g(H_{x_i}) - \sum_{j=1}^m g(F_{e_j}) + 1 - \chi(\mathcal G').$$

For any new submanifold that is included in the amalgamation at a subsequent stage, the above relationship is preserved. That is, suppose that $M' = V' \cup_{H'} W'$ has already been obtained as above by amalgamating component Heegaard splittings, and suppose $M_y = V_y \cup_{H_y} W_y$ is a submanifold and Heegaard splitting being newly amalgamated to $V' \cup_{H'} W'$ along surfaces $F_{e_1'}, \ldots, F_{e_{m'}'}$. Let $\mathcal G'$ and $\mathcal G_{y}'$ be the dual graphs for $M'$ and $M' \cup M_y$, respectively. Repeating the above argument implies that the genus of the resulting amalgamation of $M' \cup M_y$ increases by
$$g(H_y) - \sum_{k=1}^{m'} g(F_{e_k'}) + m' - 1.$$
\noindent Note that $m'$ is the number of edges of $\mathcal G_y' \setminus \mathcal G'$, and so $m' - 1 = -\chi(\mathcal G_y' \setminus \mathcal G')$. In particular, this means that $m - n + m' - 1 = 1 - \chi(\mathcal G') -\chi(\mathcal G_y' \setminus \mathcal G') = 1 - \chi(\mathcal G_y').$ Thus, the resulting genus of the amalgamation of $M' \cup M_y$ is
$$\sum_{x \in \mathcal G_y'} g(H_x) - \sum_{e \in \mathcal G_y'} g(F_e) + 1 - \chi(\mathcal G_y').$$

\noindent Amalgamating thusly along all remaining submanifolds $M_{y'}$, $y' \in \mathcal G$, produces the desired result.

\end{proof}

It is important to note that one can find examples of (minimal genus) Heegaard splittings of 3-manifolds that are not amalgamations. For example, by gluing the bridge surface of a tunnel number $n-1$, $n$-bridge knot complement to vertical annuli in a Seifert fibered space over a disk with $n$ exceptional fibers, one can obtain a Heegaard surface of the resulting 3-manifold of genus $n$, whereas the minimal genus amalgamation along the gluing surface has genus $2n$. (See~\cite{Schultens-Weidmann}.) Note that this Heegaard surface results from a very specific gluing map between the boundary components of the two submanifolds. In general, gluing maps between boundary components can be chosen to be ``sufficiently complicated'' to ensure that all minimal genus Heegaard splittings are amalgamations along the gluing surfaces. (See the appendix.) Exploiting this property in the next sections allows us to ensure that the minimal genus Heegaard splittings of our constructed 3-manifolds $M_Q$ are amalgamations, to which we can thus apply the results of this section.

\section{Constructing $M_Q$}
\label{s:constructing}

In this section we give a recipe for producing $M_Q$ from $Q$ and prove the following result.

\begin{pro}
\label{p-q-mq}
Let $Q$ be an instance of {\sc CNF-SAT}.  Then there is a manifold $M_Q$ with Heegaard genus $g(M_Q) \geq \LQ + 2$, with equality  holding if and only if $Q$ has a satisfying assignment.
\end{pro}

Recall that \LQ~is the length of $Q$ without counting parentheses.


\subsection{Constructing $M_Q$}


The sentence $Q$ will guide our construction of $M_Q$. To begin, rewrite $Q$ by inserting parentheses, if necessary, to make it clear how each logical connective joins exactly two terms (i.e.~$Q$ is made fully parenthesized). The manifold $M_Q$ is then constructed out of building blocks according to instructions provided by this modified version of $Q$. Each building block will have Heegaard genus 2 and some number of torus boundary components. Each such boundary component will be labelled with a subsentence of $Q$, and also be designated as either an {\it input} or an {\it output} to that block. We will depict such blocks so that the input boundary component is on top, and the outputs are on the bottom. See Figure \ref{f:building_blocks}. Each block is chosen based on a desired bipartitioning of its boundary components by genus 2 Heegaard splittings as follows.

\begin{figure}
\psfrag{a}{$a$}
\psfrag{b}{$\lnot a$}
\psfrag{B}{$B$}
\psfrag{A}{$A$}
\psfrag{C}{$A \land B$}
\psfrag{D}{$A \lor B$}
\psfrag{Q}{$Q$}
\psfrag{v}{\bvar}
\psfrag{r}{\brep}
\psfrag{n}{\bnot}
\psfrag{d}{\band}
\psfrag{o}{\bor}
\psfrag{e}{\bend}
\includegraphics[width=4in]{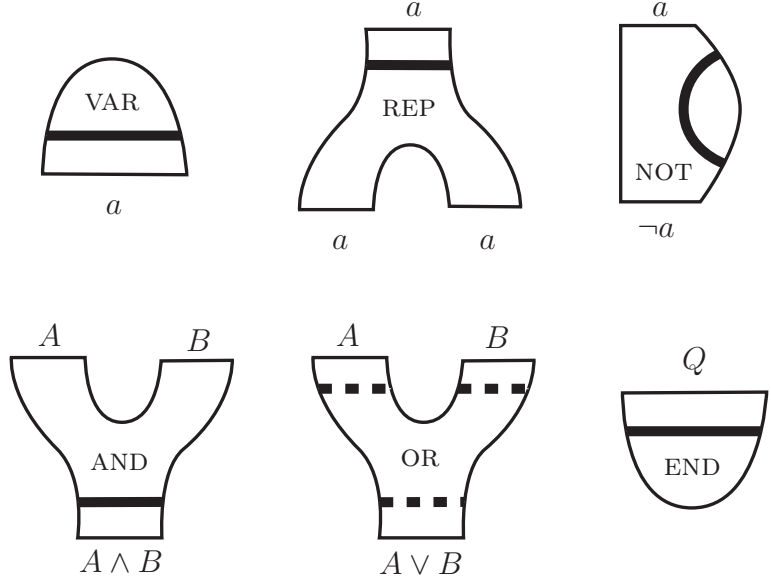}
\caption{Schematics indicating block types and their labelings. Input surfaces are depicted at the top of each block, and outputs at the bottom. Minimal genus Heegaard surfaces are depicted with bold lines (with the three possible such splittings of the \bor~ block indicated with bold dashed lines).}
\label{f:building_blocks}
\end{figure}

\begin{itemize}
		\item \bvar(iable) -  For each distinct variable in $Q$ let the block manifold $M$ be a trefoil knot exterior (Figure \ref{f:trefoilChain}).  Then  $M$ has one torus boundary component, $\partial M = T$, and any genus 2 Heegaard splitting induces the only boundary bipartition possible (up to ordering), $\{T | \emptyset\}$.  We label the boundary component $T$ with the corresponding variable, and consider it an output of the block.

\item \brep(licate) - To create multiple copies of a given variable, we use a block manifold $M$ that is the exterior of the twisted torus link in Figure \ref{f:t7176}.   Then $M$ has three torus boundary components, $\partial M = T_0 \cup T_1 \cup T_2$ where any genus two Heegaard splitting induces the boundary bipartition $\{ T_0, T_1 | T_2\}$ (Lemma \ref{l-twisted-link}).   All three components will be labelled with the variable that is being duplicated.  We will say the boundary component $T_2$ is {\it preferred}, and will be the input. The other two boundary components are outputs.

\item \bnot~ - For each occurrence of ``$\neg a$" in $Q$, the block manifold $M$ will be a high distance filling on the twisted torus link as described in Lemma \ref{l-twisted-link-filled}.  Then $\partial M = T_0 \cup T_1$ and any genus two Heegaard splitting induces the bipartition $\{ T_0,T_1 | \emptyset \}$.  Label one boundary component $a$, and consider it an input. The other boundary component is labelled $\lnot a$ and is considered an output. Glue the input surface to the output of a \brep~block corresponding to $a$.
\end{itemize}

Once we have created  one labeled output surface for each instance of each variable in $Q$, and each instance of its negation, we start gluing them to other kinds of blocks determined by the logical structure of $Q$, as follows:	

\begin{itemize}

\item \band~- For each conjunction $A \land B$ in $Q$, we let $M$ be the exterior of the twisted torus link already used for \brep. Then $\partial M = T_0 \cup T_1 \cup T_2$, and all genus two Heegaard splittings all induce the bipartition $\{ T_0,T_1 | T_2\}$ (Lemma \ref{l-twisted-link}). Label the preferred boundary component $T_2$ with the expression $A \land B$, and consider it an output. The other two boundary components are inputs, and are labelled with the expressions $A$ and $B$ respectively.

\item \bor~- For each disjunction $A \lor B$ in $Q$ we let $M$ be the exterior of the three component chain indicated in Figure \ref{f:trefoilChain}.  It is homeomorphic to $\textit{\{pair of pants\}} \times S^1$, has three boundary components $\partial M = T_0 \cup T_1 \cup T_2$, and each of the three boundary bipartitions of the form $\{ T_i, T_j | T_k\}$ is realized by some genus two Heegaard splitting (Lemma \ref{l-pair-of-pants}).  Choose one boundary component to label as $A \lor B$, and consider it an output. The remaining boundary components are inputs, and are labelled with the expressions $A$ and $B$ respectively.

\item \bend~-  We end by capping the statement off with the same $M$, the trefoil knot exterior, used for \bvar.  The manifold $M$ has one torus boundary component, $\partial M = T$, and a single boundary bipartition $\{T | \emptyset\}$.  It is labelled with the entire expression $Q$, and is an input.
\end{itemize}


To glue the blocks, we choose ``sufficiently complicated" maps so that every Heegaard splitting of $M_Q$ of genus less than or equal to $\LQ+2$ is an amalgamation of splittings of the blocks. (See the appendix.)

As an example, Figure \ref{f:MQexample} gives the construction of the manifold $M_Q$ from the expression \[Q=((a \lor c) \land (\lnot a \lor b)) \land (b \lor c).\]



\subsection{Proof of Proposition \ref{p-q-mq}.}

\begin{lem}
\label{l:GenusMin}
The Heegaard genus of $M_Q$ is at least $\LQ+2$, and in the case of equality, any such minimal genus splitting is an amalgamation of minimal genus splittings of the building blocks.
\end{lem}

\begin{proof}
Let $S$ be a minimal genus Heegaard splitting of $M_Q$. If the genus of $S$ is strictly greater than $\LQ+2$ then the result follows. By way of contradiction, we assume the genus of $S$ is at most $\LQ+2$. By construction, $S$ is then an amalgamation of Heegaard splittings of the building blocks. We now use Theorem \ref{t:GenusCount} to compute the genus of $S$:

$$g(S) =\displaystyle{\sum_{x \in \mathcal G} g(H_x) - \sum_{e \in \mathcal G} g(F_e) + 1 - \chi(\mathcal G)}.$$

Here $\mathcal G$ is the graph dual to the block structure.  Let $v$ be the number of vertices, one for each block, and $e$ the number of edges, one for each gluing torus.  Note that the number of variable occurrences in $Q$ is the number of \bvar~and \brep~blocks. The operators in $Q$ each have a corresponding \bnot, \bor, or \band~block, and there is a final \bend~ block for the total statement $Q$. In particular, $v = \LQ + 1$. Since each block has genus 2, we have $g(H_x) \geq 2$ for each $x$, with equality holding only for those blocks with minimal splittings, and $g(F_e)=1$ for each $e$.  Thus,

$$g(S) \geq  2v - e  + 1 - (-e+v) = v +1 = \LQ + 2.$$
\end{proof}

\begin{lem}
\label{l-min-sat}
If the Heegaard genus of $M_Q$ is equal to $\LQ+2$ then there is a satisfying assignment of $Q$.
\end{lem}

\begin{proof}
Suppose $S$ is a minimal genus Heegaard surface of $M_Q$. If the genus of $S$ is $\LQ+2$, then by the previous lemma $S$ is an amalgamation of minimal genus Heegaard surfaces $\{S_i\}$ in the building blocks.

Because $S$ is an amalgamation, the surfaces $\{S_i\}$, together with the gluing surfaces, separate the manifold $M_Q$ into compression bodies that can be colored ``black" and ``white" so that no two compression bodies with the same color are adjacent. Without loss of generality, we assume the compression body of the \bend~ block which contains its sole input surface is colored white.

We will now assign truth values to the gluing surfaces between blocks, according to this bicoloring. Let $F$ be such a gluing surface. Then $F$ is the input surface for some block. If the compression body in that block containing $F$ is white, then we will say that $F$ is {\it true}. Otherwise, we say it is {\it false}. Equivalently, we can say that $F$ is {\it true} if it is the output of a block, and the compression body in that block that contains $F$ is black. Thus, if the Heegaard surface in some block separates an input surface $A$ of that block from an output surface, $B$, then $A$ and $B$ will have the same truth value. It follows immediately that the input and output surfaces of all \brep~ blocks have the same truth value. Similarly, the truth value of  the input of a \bnot~ block labelled $a$ will have the opposite truth value as the output labelled $\lnot a$. Finally, note that the surface at the input of the \bend~ block (which we have labelled with the statement $Q$) is by choice assigned the truth value {\it true.}

In the next several claims, we show that our assignment of truth values respects the logical structure of the subsentences of $Q$ that appear at the labels of (most of) the gluing surfaces.

\begin{clm}
All surfaces at the inputs and outputs of the \band~ blocks are {\it true}.
\end{clm}

\begin{proof}
The minimal genus Heegaard surface of an \band~ block separates the output surface from both inputs. Thus, the output  and input surfaces all have the same truth value. The proof is complete by noting that since $Q$ is in conjunctive normal form, the output of every \band~ block is glued to the input of the \bend~ block (a {\it true} surface), or the input of another \band~ block.
\end{proof}

We say an {\it \bor-tree} is a component of the union of the \bor~ blocks in $M_Q$.

\begin{clm}
The output of every \bor-tree is {\it true}, and at least one of the input surfaces of every \bor-tree is {\it true}.
\end{clm}

\begin{proof}
Let $F_0$ denote the output surface of an \bor-tree. Since $Q$ is in conjunctive normal form, $F_0$ is glued to the input of an \band~ block. By the previous claim, $F_0$ must be {\it true.} By construction, the Heegaard surface of the \bor~ block that contains $F_0$ separates it from at least one of the input surfaces $F_1$ of that block. Thus, $F_1$ will also be {\it true}. Working up the tree, we now consider the \bor~ block in the tree whose output is the surface $F_1$. By identical reasoning, one of its input surfaces $F_2$ must be {\it true} as well. Continuing in this way we eventually reach an input surface $F_i$ of the entire \bor-tree and conclude that it must be {\it true}.
\end{proof}

Note that some of the truth values of the sentences that label gluing surfaces interior to an \bor-tree may not be correct, but the previous claim shows this does not disturb the logical structure of the \bor-tree, taken as a whole.

To complete the lemma, note that we have assigned a truth value to the output surface of every \bvar~ block. These surfaces correspond to the variables used in the sentence $Q$. We have shown above that our assignment of truth values to the input and output surfaces of  \brep, \bnot, and \band~ blocks, as well as \bor~ trees, respects the logical structure of the sentences that label them. Thus, we have produced an assignment of truth values for the variables that make the statement $Q$ {\it true}.
\end{proof}

\begin{lem}
\label{l-sat-min}
If there is a satisfying assignment of $Q$, then the Heegaard genus of $M_Q$ is equal to $\LQ+2$.
\end{lem}

\begin{proof}
If there is a satisfying assignment of $Q$, then that assignment gives a truth value to each expression at the gluing surfaces. In this way, each boundary component of each building block gets assigned a truth value. We color the sides of each such surface black/white so that if $F$ is a {\it true} surface at the output of a block, then the side of $F$ facing into that block is black. Similarly, if $F$ is a {\it true} surface at the input of a block, then the side facing in is colored white. Conversely, the side of a {\it false} surface at the output of a block is colored white, and the side of a {\it false} surface at an input is black.

\begin{clm}
\label{c:2color}
There is a minimal genus splitting of each block that separates all white surfaces on the inside of the block from all black surfaces facing in.
\end{clm}

\begin{proof}
Consider first the \bend~ block. Since there is only one boundary component, any Heegaard splitting (and in particular the minimal genus one) has the desired separation property.

Next we consider the \band~ blocks. Since $Q$ is in conjunctive normal form, the output of each such block is either attached to the \bend~ block, or another \band~ block. Hence, if there is a satisfying assignment for $Q$ then the labels at every input and output surface of an \band~ block are {\it true} logical sentences. It follows that the side of the input surfaces that face into such a block are white, and the side of the output surface facing into the block is black. Such a block has the output as a preferred boundary component, meaning that a minimal genus splitting separates the output surface from both input surfaces. Hence, the minimal genus splitting has the desired separation property.

An \bor~ block has no preferred boundary component. Thus, there is a minimal genus splitting for each non-trivial bipartitioning of the boundary components. It follows that the only way the separation property can fail is if the side of every boundary surface facing in to the block is the same color. If they are all white, then this corresponds to both inputs being {\it true}, and the output being {\it false}. If they are all black, then both inputs are {\it false}, and the output is {\it true}. Neither situation obeys the properties of the logical ``or" operation, so we will not see these sets of truth values for the labels of the surfaces at the boundary of an \bor~ block.

By construction, a \brep~ block has the same logical value at each input and output. If they are all {\it true}, then the side of the input surface that faces into the block is white, and the side of the outputs that faces in is black. The input surface of this block is a preferred boundary component, so the minimal genus splitting separates black from white as desired. If all surfaces are {\it false}, the situation is reversed.

Finally, we consider the \bnot~ blocks. The sentences at the boundary components of a \bnot~ block will have opposite truth values. Thus, the side of the input surface facing into the block will have the same color as the side of the output surface facing in. Both surfaces are on the same side of a minimal genus splitting of a \bnot~ block.
\end{proof}

Assume we have now chosen splittings of each block in accordance with the conclusion of Claim \ref{c:2color}. Then the building blocks are separated into compression bodies by these splittings, and these compression bodies inherit the color black or white, according to the colors of their negative boundaries. Furthermore, because opposite sides of any single gluing surface are different colors, it follows that neighboring compression-bodies in $M_Q$ are colored differently.

According to Theorem \ref{thm:GHS}, to show that we can amalgamate our choice of splittings of the building blocks, it remains to show that the directed graph $\mathcal G$ that is dual to the gluing surfaces has no directed cycles. (Recall that each edge of this graph is oriented so that it passes from a black compression body into a white one.)

We have constructed $M_Q$ vertically so that the output surface(s) of any given block is below its input surface(s). Any directed cycle must have a local maximum, $x$. Let $e_1$ and $e_2$ be the edges of the cycle that meet $x$, where $e_1$ is oriented toward $x$, and $e_2$ is oriented away. As $x$ is a local maximum, both $e_1$ and $e_2$ correspond to output surfaces of the building block corresponding to $x$. It follows that this building block is a \brep~ block, as this is the only type of block that has two output surfaces. However, according to our coloring scheme, both output surfaces of a \brep~ block are on the boundary of the same compression body. If this compression body is black, then both $e_1$ and $e_2$ are  oriented away from $x$. If the compression body is white, then both are oriented toward $x$. This contradiction establishes that there are no directed cycles in $\mathcal G$.

By Theorem \ref{thm:GHS} we can now amalgamate the chosen splittings of our building blocks, creating a splitting of $M_Q$. By the computation given in the proof of Lemma \ref{l:GenusMin}, the genus of this splitting is $\LQ+2$.
\end{proof}

Finally, note that if one were to remove the \bvar~ blocks from $M_Q$, we would obtain a manifold with a boundary component corresponding to each variable, and, for each satisfying assignment,  a minimal genus Heegaard splitting that induces a \{{\it true} $|$ {\it false}\} bipartition of the corresponding boundary components.  That is the basis for the following corollary.

\begin{cor}
\label{cor-partitions}
Let $\mathcal P$ be a non-empty set of bipartitions of  $1,2,...,n$.  Then there is a 3-manifold $X$ and a numbering of its boundary components, $1,2,...,n$,  so that the set of bipartitions of $\partial X$ induced by minimal genus Heegaard splittings of $X$ is precisely $\mathcal P$.
\end{cor}

\renewcommand{\P}{\mathcal P}

\begin{proof}
Suppose that $P$ is a bipartition of $1,..,n$. That is, $P = \{ P_+ | P_-\}$ so that $P_+ \cup P_- = 1,..,n$ and $P_+ \cap P_- = \emptyset$.  Let $v_i,i=1,..,n$ be variables and let the clause $q(P)$ be a conjunction of each variable or its negation, depending on which side of the bipartition $P$ its index belongs to:

$$q(P) = \bigwedge \{v_i|i \in P_+\} \bigwedge \{\neg v_i|i \in P_-\}.$$

Of course, $q(P)$ accepts exactly one satisfying assignment, and that corresponds (via the correspondence $i \in P_+ \iff v_i = \textit{true}$) to the bipartition $P$.   Now let $\P$ be a set of bipartitions of $1,...,n$ and let $\P^C$ be its complement, i.e.~the set of bipartitions not in $\mathcal P$.    Let

$$Q(\P^C) = \bigvee \{ q(P) | P \in \P^C \}$$

Now, let $Q = Q(\P) = \neg Q(\P^C)$ which, after applying De Morgan's laws, is an instance of {\sc CNF-SAT}.  Let $M_Q$ be built according to the procedure above.   Now it is easy to check that satisfying assignments are in 1-1 correspondence with bipartitions $P \in \P$, again by using the correspondence $i \in P_+ \iff v_i=\textit{true}$.

Let $M_Q$ be constructed as before.  Note that since $Q$ is satisfiable, $M_Q$ has Heegaard genus $\LQ + 2$.  Let $M_Q'$ be the manifold obtained by removing each \bvar~ block.  Because each \bvar~ block removed is a leaf in $\mathcal G$, the graph dual to the block structure, the proofs of Lemmas \ref{l-min-sat} and \ref{l-sat-min} apply to $M_Q'$ as well as to $M_Q$.   In particular,  a minimal genus splitting of $M_Q'$ determines a satisfying truth assignment to the $v_i$'s, and vice-versa.   Note that each $v_i$ labels a boundary component of $M_Q'$, and each minimal genus splitting  separates the {\it true} variables from the {\it false} variables, so bipartitions induced by minimal genus splittings are in 1-1 correspondence with satisfying assignments which in turn are in 1-1 correspondence with bipartitions $P \in \P$ (via $i \in P_+ \iff v_i= \textit{ true}$).

\end{proof}

\section{Triangulating $M_Q$}
\label{s:triangulating}

\newcommand{\F}{\mathcal F}
\newcommand{\D}{d_\F}
\newcommand{\R}[2]{\tfrac{F_{#1}}{F_{#2}}}

In this section, we describe how to triangulate the manifold $M_Q$ so that the number of tetrahedra used is at most quadratic in $|Q|$, the length of the statement $Q$.   Our goal is the following:

\begin{pro}
\label{t-triangulation-quadratic}
A triangulated $M_Q$ can be produced in quadratic time (and tetrahedra) in \LQ.
\end{pro}

We proceed in several steps.  First, in Sections \ref{s:Farey} and \ref{s:layering} we give a method to perform high distance triangulated gluings via {\it layered triangulations}.  For the most part, these are not new results.  Our statements about distances in the Farey graph in Section \ref{s:Farey} are certainly well known, and layered triangulations (Section \ref{s:layering}) are described by Jaco and Rubinstein in \cite{jaco-rubinstein-layered}.   We include these sections, instead of just citing earlier work, because they are both accessible to the non-expert and also make explicit the relationship between the distance of the gluing and the number of layers.

Next, in Section \ref{s:blocks}, we give a topological description of block manifolds whose boundary components are appropriately bipartitioned by minimal genus Heegaard splittings.  We consolidate some well known results and  substantially leverage the work of Morimoto, Sakuma, and Yokota on Heegaard splittings of twisted torus knots \cite{msy}, and the work of Moriah, Rieck, Rubinstein and Sedgwick that characterizes how and when a Dehn filling creates new Heegaard splittings \cite{moriah-rubinstein, rieck-heegaard-dehn,rieck-sedgwick-1,rieck-sedgwick-2,moriah-sedgwick-heegaard-dehn}.

We conclude, in Section \ref{s:proof_prop}, with a proof of Proposition \ref{t-triangulation-quadratic} that describes how the blocks can be triangulated and then glued together.


\subsection{Slopes and the Farey Graph}
\label{s:Farey}

A {\em slope} is the isotopy class of an essential simple closed curve on a torus.  Fix a pair of basis elements for the homology, $\mathbb Z \times \mathbb Z$, of the torus.  Then any slope can be written as a pair $(a,b)$, and because it is realized by a simple (connected) curve, we have $\gcd(a,b)=1$.  The usual convention is thus to represent the slope by the extended rational $\frac{a}{b} \in \mathbb Q \cup \{\infty\}$, where $\infty = \frac{1}{0}$.

\begin{figure}
\includegraphics[width=4in]{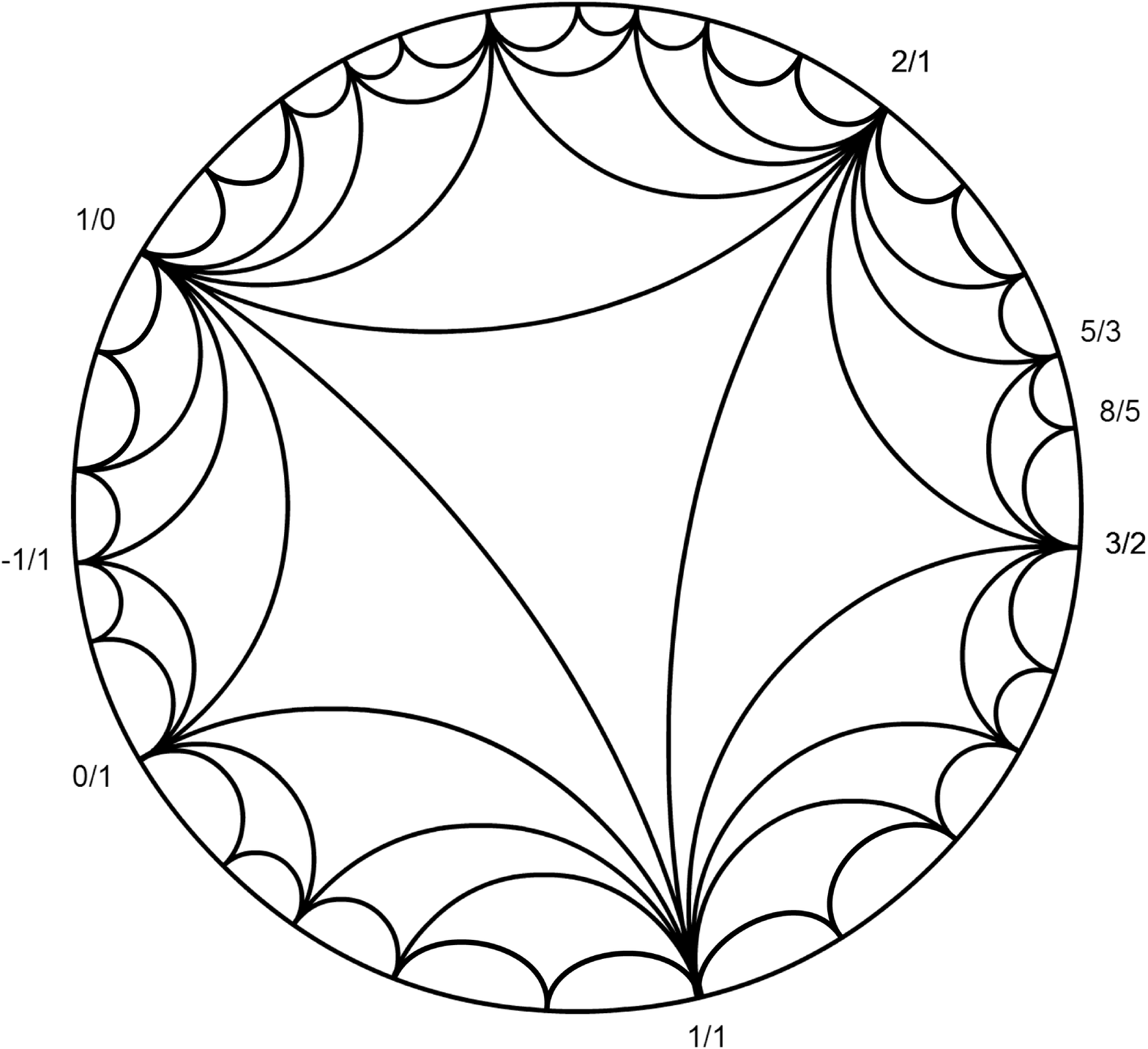}
\caption{The Farey Tessellation of the Poincar\'e Disk.}
\label{f-farey-graph}
\end{figure}

We say that a pair of slopes have {\em distance one} if there are a pair of curves representing the slopes that intersect transversely in a single point.  It is well known that a pair of slopes have distance one if and only if  their extended rationals (with respect to any basis), $\frac{a}{b}$ and $\frac{c}{d}$ , satisfy $|ad-bc|=1$.

\begin{dfn}
Let $T$ be a torus.  The {\em Farey graph for $T$} is the graph whose vertex set is the set of slopes and whose edges join any pair of vertices whose underlying slopes have distance one.   Of course, after choosing a basis for homology, we are able to label each vertex of the graph with an extended rational $\frac{a}{b} \in \mathbb Q \cup \{\infty\}$.  Each edge then joins a pair of extended rationals, $\frac{a}{b}$ and $\frac{c}{d}$,  which satisfies $|ad-bc|=1$.
\end{dfn}

\begin{dfn}
If $\alpha$ and $\beta$ are slopes in a torus $T$, then the {\em Farey distance} between them $d_\F(\alpha,\beta)$ is their distance in the Farey graph.  If $a \subset T$ and $b \subset T$ are closed essential curves, then we define their distance, $d_\F(a,b) = d_\F(\alpha,\beta)$, to be the distance between $\alpha$ and $\beta$, isotopy classes of single components of $a$ and $b$, respectively.
\end{dfn}

Form a $2$-complex, the {\em curve complex of the torus $T$}, by attaching to the Farey graph a triangular face for every triple of slopes that pairwise intersect once.   Fixing a basis for $T$, every edge is specified by a pair $\left(\frac a b, \frac c d \right)$ satisfying $|ad-bc|=1$.  It is not hard to see that in the curve complex, there are precisely two triangles, $\left(\frac a b, \frac c d, \frac {a+c}{b+d} \right)$ and $\left(\frac a b, \frac c d, \frac{a-c}{b-d} \right)$ attached to the edge $\left(\frac a b, \frac c d \right)$.  This is described by the well known {\em Farey tessellation} of the Poincar\'e disk model of $\mathbb H^2$, see Figure \ref{f-farey-graph}.

Moreover, each triangular face identifies a triangulation of the torus $T$ up to isotopy:   The slopes $\frac a b$ and $\frac c d$ can be realized by a pair of curves in the torus meeting in a single point.   Together, they cut the torus into a rectangle.   This rectangle has exactly two choices for a diagonal curve, with slopes $\frac {a+c}{b+d}$ and  $\frac {a-c}{b-d}$ when connected through the intersection point.  Choose one, say $\frac {a+c}{b+d}$.  Then the triple of curves $\left(\frac a b, \frac c d, \frac {a+c}{b+d}\right)$ intersect in a single common point.  Treating that point as a vertex, we have formed a (non-simplicial) triangulation of the torus $T$ with one vertex, three edges and two faces.   We call this a {\em one-vertex triangulation of the torus}.  Note that the two triangulations $\left(\frac a b, \frac c d, \frac {a+c}{b+d}\right)$ and $\left(\frac a b, \frac c d, \frac{a-c}{b-d}\right)$ meeting the edge $\left(\frac a b, \frac c d\right)$ are related by a {\em diagonal flip}, that exchanges the diagonal $\frac {a+c}{b+d}$ for the diagonal $\frac {a-c}{b-d}$, or vice-versa.

\subsection{Layering}
\label{s:layering}

Later we will assume that our manifold $X$ has been endowed with a triangulation that restricts to a one vertex triangulation of each of its torus boundary components \cite{jaco-rubinstein-0efficient}.

\begin{figure}[h]
\psfrag{o}{$e$}
\psfrag{n}{$e_\Delta'$}
\psfrag{d}{$e_\Delta$}
\psfrag{p}{$+$}
\psfrag{e}{$=$}
\includegraphics[width=3.5in]{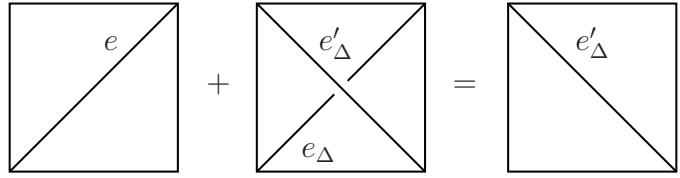}
\caption{Layering a tetrahedron on the boundary swaps a diagonal.}
\label{f:layer}
\end{figure}

Let $e$ be an edge in the triangulation of the boundary torus $T \subset \partial X$.   Then $e$ can be regarded as the diagonal of a rectangle $R$ bounded by the other two edges.   Picture a new tetrahedron, $\Delta$, as being a slightly thickened horizontal rectangle.  Its bottom is a rectangle $R_\Delta$ with diagonal $e_\Delta$ and its top is a rectangle $R_\Delta'$ with diagonal $e_\Delta'$.  See Figure~\ref{f:layer}.  One can form a new triangulated manifold $X' = X \cup_{R=R_\Delta} \Delta$, by gluing $R$ to $R_\Delta$ so that the diagonals $e$ and $e_\Delta$ are identified.  This process is called {\em layering at $e$}   (see also \cite{jaco-sedgwick}).  It is not hard to see that the manifold $X'$ is homeomorphic to $X$ (as it retracts onto $X$) but that the boundary triangulation has changed.  In particular, while $e$ is no longer in the boundary torus, the boundary of $R$ is still in the boundary torus, but its diagonal is now opposite and realized by $e_\Delta'$.  Thus, layering at $e$ performs a diagonal flip on $e$ in the boundary triangulation.  The two triangulations are represented in the Farey tessellation by a pair of triangles that share a common edge.

\begin{lem}
\label{l-layer-fib}
Let $T \subset \partial X$ have a one-vertex triangulation  with edge slopes $\left(\frac 0 1, \frac 1 0, \frac 1 1\right)$.  Then, by layering on $k$ tetrahedra, we can obtain a new triangulation of $X$ with edge slopes $\left( \R{k-1}{k-2}, \R{k}{k-1}, \R{k+1}{k}\right)$, where $F_k$ is the $k^{th}$ Fibonacci number.
\end{lem}

\begin{proof}
Consider the sequence $\frac 0 1, \frac 1 0, \frac 1 1, \frac 2 1, \frac 3 2, \frac 5 3,  \ldots,  \R{k-1}{k-2}, \R{k}{k-1}, \R{k+1}{k}$.   Note that each successive triple of terms determines a triangulation, and that each successive pair of triples share two slopes. Hence, the latter boundary triangulation can be obtained by layering on the edge of the former that they do not share.  It takes $k$ steps, hence $k$ layers, to move from the first triple to the last.
\end{proof}

Furthermore, continued layering in this fashion increases the distance between the latest edge slopes and the original edge slopes:

\begin{lem}
\label{l-distance-fib}
Let $F_k$ be the $k^{th}$ Fibonacci number.  Then,
$$\D\left(\tfrac{F_{k+1}}{F_k},\infty\right) = \lfloor k/2 \rfloor + 1    $$
\end{lem}

\begin{proof}
We will give an inductive proof.   It is easy to verify that the statement holds for $k=0,1,2$, where $\R{k+1}{k}=\tfrac{1}{1},\tfrac{2}{1},\tfrac{3}{2}$, respectively, and the distances to $\infty = \tfrac{1}{0}$ are $1,1,2$, respectively.  Let $k$ be the least $k$ for which the conclusion of the lemma does not hold.  In the Poincar\'e disk, consider the triangle $\left(\R{k-1}{k-2},\R{k}{k-1},\R{k+1}{k}\right)$ which is bounded by edges of the Farey Graph (see Figure \ref{f-farey-graph}). This triangle separates the disk into 3 components.

First, we claim that the points $\R{k+1}{k}$ and $\infty = \frac{1}{0}$ lie on opposite sides of the edge $\left(\R{k}{k-1},\R{k-1}{k-2}\right)$.  To see this, note that the point $\R{k-2}{k-3}$ is the other corner of the second triangle that meets the edge $\left(\R{k}{k-1},\R{k-1}{k-2}\right)$.  The inductive hypothesis implies $\D\left(\R{k-2}{k-3},\infty \right) < \D\left(\R{k}{k-1},\infty \right)$, so the second triangle must lie on the same side of the edge $\left(\R{k}{k-1},\R{k-1}{k-2} \right)$ as $\infty$, hence the point $\R{k+1}{k}$, lies on the other side.

Now, take a minimal path in the Farey Graph joining $\infty$ to $\R{k-1}{k-2}$.  By adjoining the edge $\left(\R{k-1}{k-2},\R{k+1}{k}\right)$ to that path, we obtain a path from $\infty$ to $\R{k+1}{k}$.   It follows that $\D\left(\R{k+1}{k},\infty \right) \leq \D \left(\R{k-1}{k-2},\infty\right)+1$.

Now, take a minimal path from $\infty$ to $\R{k+1}{k}$.  Because $\infty$ and $\R{k}{k-1}$ lie on opposite sides of the edge $\left(\R{k-1}{k-2},\R{k}{k-1} \right)$, this minimal path must pass through either the point $\R{k-1}{k-2}$ or the point $\R{k}{k-1}$.   It follows that
\begin{eqnarray*}
\D\left(\R{k+1}{k},\infty \right) &\geq& \min \left\{ \D\left(\R{k-1}{k-2},\infty\right)+1,\D\left(\R{k}{k-1},\infty \right)+1\right\}\\
& =& \D \left(\R{k-1}{k-2},\infty \right)+1.
\end{eqnarray*}

Thus, $\D \left(\R{k+1}{k},\infty \right) = \D\left(\R{k-1}{k-2},\infty\right)+1$ and the desired result follows.

\end{proof}

\begin{lem}
\label{l-triang-gluing}
Let $X$ be a (possibly disconnected) 3-manifold given via a  triangulation that has a single vertex in each of two torus boundary components, $T_0$ and $T_1$.  If $\alpha_0 \subset T_0$ and $\alpha_1 \subset T_1$ are slopes and $D \in \mathbb N$, then there is a triangulated manifold $X'$ obtained from $X$ by gluing $T_0$ to $T_1$ so that
\begin{itemize}
\item $\D(\alpha_0,\alpha_1) >  D$, where distance is measured in the common image of $T_0$ and $T_1$ in $X'$, and
\item $t(X') = t(X) + 2D$, where $t(\cdot)$ is  number of tetrahedra.
\end{itemize}
\end{lem}

\begin{proof}
Fix an orientation on $X$ and assume that the $T_i, i=0,1$, have the induced boundary orientation.   For each $i=0,1$, we may choose a basis, $(0,\infty)$, for the homology of the boundary torus $T_i$ so that the edges of the one-vertex triangulation have slopes $(0,\infty,1)$, the basis $(0,\infty)$ induces the boundary orientation, and $\alpha_i$ has non-positive slope, $\alpha_i \leq 0$.

Applying Lemma \ref{l-layer-fib}, layer $2D$ tetrahedra on the boundary component $T_0$ so that the resulting triangulation has edges with slopes $\left( \R{2D-1}{2D-2}, \R{2D}{2D-1}, \R{2D+1}{2D} \right)$.

Now, let $X'$ be the manifold obtained by gluing the boundary triangulations together via an orientation reversing map that identifies the edge with slope $\R{2D+1}{2D}$ in $T_0$ with the edge with slope $0$ in $T_1$.  This identifies the pair of edges with slopes $\left( \R{2D-1}{2D-2}, \R{2D}{2D-1} \right)$ in $T_0$, with the pair of edges with slopes $(1,\infty)$ in $T_1$, or its reverse.  Note that the edge $\left( \R{2D-1}{2D-2}, \R{2D}{2D-1} \right)$ in the Farey graph for $T_0$ separates $\infty$ and the image of $\alpha_1$.

Now compute the distance in the original basis for $T_0$ using Lemma \ref{l-distance-fib}.  We have distance $\D \left(\alpha_0,\alpha_1 \right) > \D \left(\infty,\R{2D-1}{2D-2} \right) = \lfloor \frac {2D-2}{2} \rfloor + 1 = D$, as claimed.

\end{proof}

\subsection{Blocks from Links}
\label{s:blocks}

In this section we construct the required block manifolds.  In each case, we prescribe a set of bipartitions of boundary components and then construct a manifold whose minimal genus Heegaard surfaces induce precisely that set of bipartitions of boundary components.   All of our examples are Heegaard genus two.   Three of the four are realized as the {\it exterior} of a knot or link in $S^3$, that is, each manifold is homeomorphic to $X(L) = S^3 - N(L)$ where $L$ is a knot or link in $S^3$ and $N(\cdot)$ denotes an open regular neighborhood.   The boundary of each manifold is a union of tori, and we often abuse notation by referring to components of the link, rather than to their corresponding boundary components.  The fourth block manifold is obtained by Dehn filling on a torus boundary component of the third block manifold. Many of the results in this section are not new, and are collected for the sake of specificity.


\begin{figure}[h]

\includegraphics[width=1.35in]{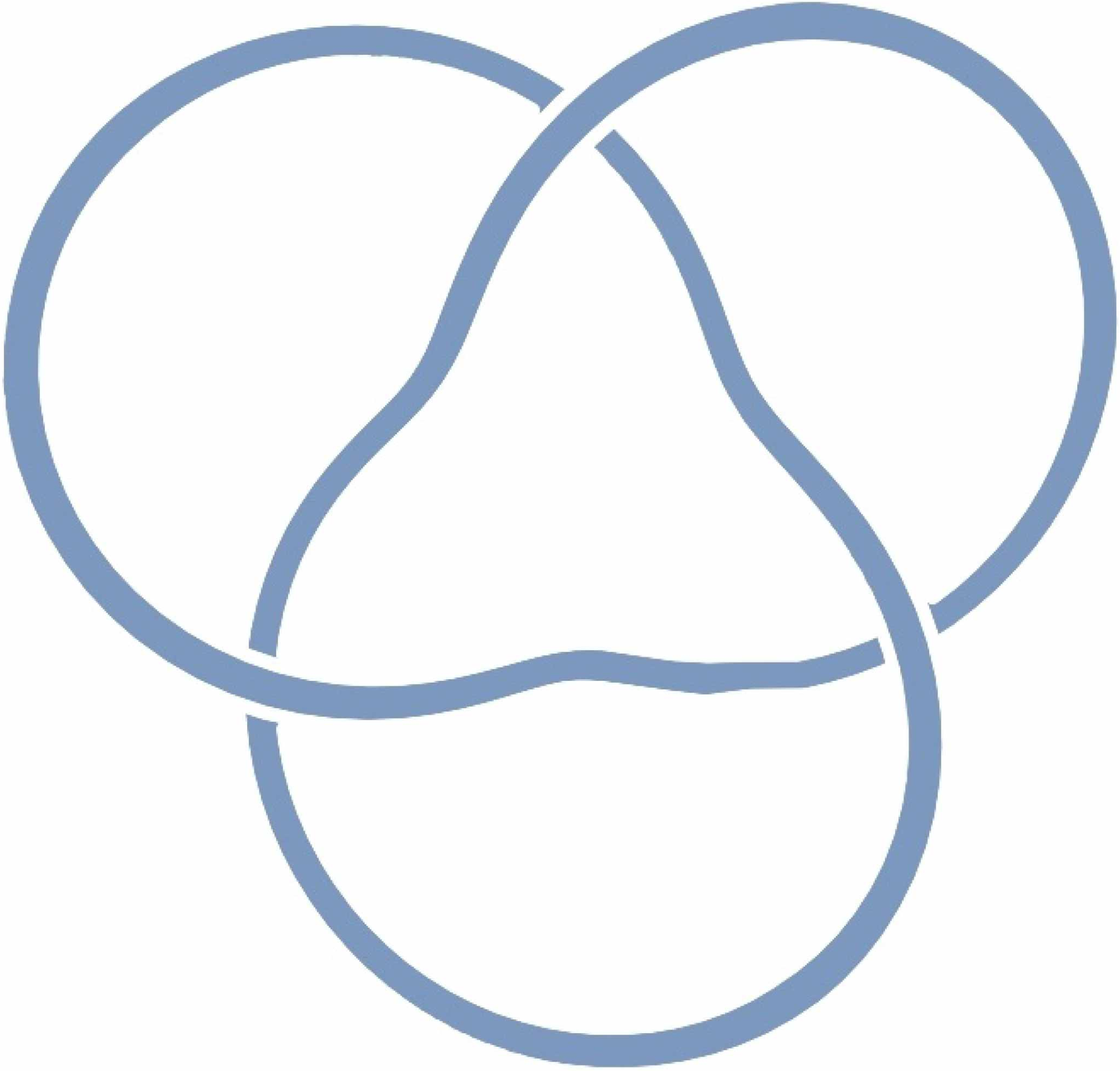}
\hspace{.5in}
\includegraphics[width=2.75in]{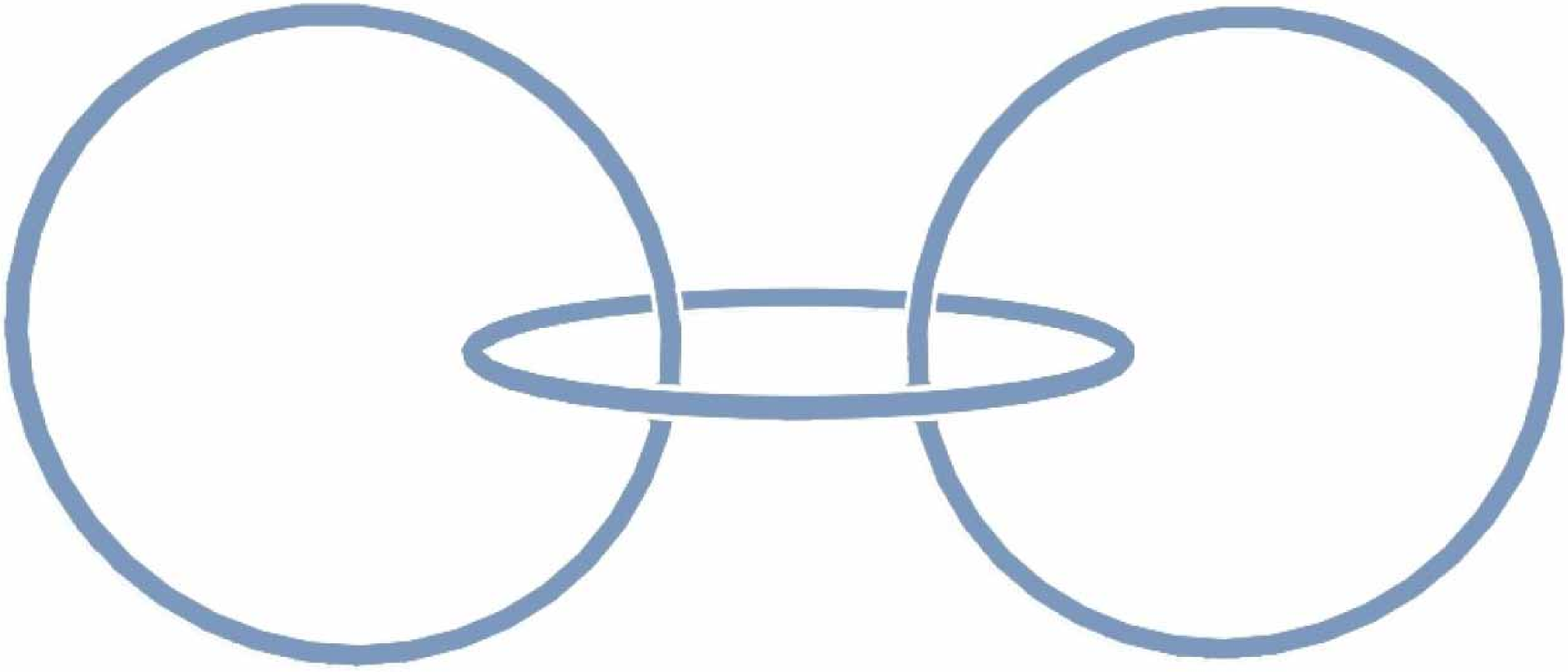}
\caption{Trefoil knot and 3 link chain.}
\label{f:trefoilChain}
\end{figure}

For \bvar~ blocks and the \bend~ block we need a genus two manifold with a single incompressible torus boundary component.  The exterior of any tunnel number one knot will do, we choose a simple one:

\begin{lem}[\bvar, \bend]
\label{l-trefoil-complement}
Let $K \subset S^3$ be the trefoil knot (see Figure \ref{f:trefoilChain}) and $X(K) = S^3 - N(K)$ be its exterior.  Then $X(K)$ has Heegaard genus two.
\end{lem}

\begin{proof} It is well known that $K$ is tunnel number one (genus two), see e.g.~\cite{kobayashi-2-bridge}.
\end{proof}

For \bor~ blocks, we want a manifold whose minimal genus Heegaard surfaces realize every non-trivial bipartition of its three boundary components.   The simplest such manifold seems to be the exterior of the three component chain,  whose irreducible, and even non-irreducible, Heegaard splittings are quite well understood \cite{schultens-Heegaard-product}, \cite{moriah-sedgwick-weakly}.  Note that it is impossible for a genus two Heegaard surface to trivially bipartition the boundary components, $\{T_0, T_1, T_2  | \emptyset \}$, as a genus two compression body $V$ cannot have three torus boundary components in $\partial_-V$.

\begin{lem}[\bor]
\label{l-pair-of-pants}
Let $C \subset S^3$ be the three component chain (see Figure \ref{f:trefoilChain}), and $X(C) = S^3 - N(C)$ its exterior.  Then,
\begin{enumerate}

\item $X(C)$ has Heegaard genus two,
\item every non-trivial bipartition $\{T_i, T_j| T_k\}$ of the three boundary components of $\partial X(C)$ is induced by a genus two Heegaard surface for $X(C)$.
\end{enumerate}
\end{lem}

\begin{proof}
Again, these facts are well known:  it is easy to see that for each pair of link components, there is a handle and a short arc connecting them that induces a genus two Heegaard splitting that separates the pair from the other link component.

\end{proof}

\begin{figure}[h]
\includegraphics[width=2.25in]{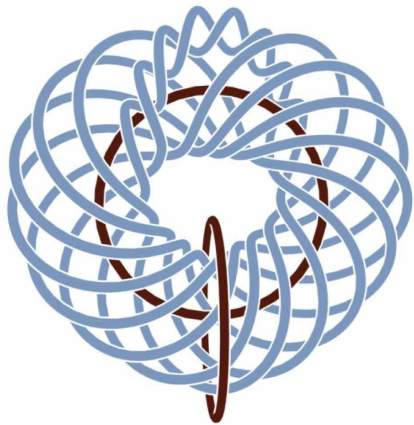}
\caption{Link with three components: $T(7,17,6)$ and two unknots $U_0$ and $U_1$.}
\label{f:t7176}
\end{figure}

For \band~ and \brep~ blocks, we want a manifold whose minimal genus Heegaard surfaces all prefer the same bipartition of its three boundary components.  This is a bit more challenging. Fortunately, Morimoto, Sakuma and Yokota showed that certain twisted torus knots are not 1-bridge with respect to an unknotted torus in $S^3$, providing the basis for the following.

\begin{lem}[\band, \brep]
\label{l-twisted-link}
Let $L \subset S^3$ be the link indicated in Figure \ref{f:t7176}.  It is the union of the twisted torus knot $T(7,17,6)$ along with two unknotted components $U_0$ and $U_1$.  Let $X(L)$ be its exterior.   Then,
\begin{enumerate}
\item $X(L)$ has Heegaard genus two,
\item any genus two Heegaard splitting of $X(L)$ induces the same bipartition of boundary components, that is  $\{U_0, U_1 |T(7,17,6) \}$,

\item $X(L)$ does not contain a M\"obius band with its boundary contained on the knotted boundary component.


\end{enumerate}
\end{lem}

Note that conclusion (3) is not  needed for the \band~ or \brep~ blocks themselves.  Rather, it is technical condition used for the construction of the \bnot~ block via Lemma \ref{l-twisted-link-filled}, which follows.

\begin{proof}
(1) It is well known \cite{msy} and easy to see that a short arc joining the pair of twisted strands is a tunnel system for $T(7,17,6)$.   The strands can be untwisted by sliding them over the tunnel, after which the tunnel appears to be the ``middle tunnel" \cite{moriah-sedgwick-twisted} for the torus knot $T(7,17)$.  Moreover, this gives a genus two splitting of the entire link as the indicated unknots $U_0$ and $U_1$ are cores for the complementary handlebody.  Note that this genus two splitting induces the bipartition $\{T(7,17,6) | U_0,U_1\}$ of the boundary components.  This is also a minimal genus splitting as no exterior of a link with 3 components has genus one.

(2) Suppose that a genus two Heegaard splitting induces a bipartition that isolates one of the two unknotted components, $\{U_i|U_j,T(7,17,6)\}$, for some $i \neq j$.  In particular, this implies that the link $T(7,17,6) \cup U_j$ is tunnel number one.   Lemma 4.13 of \cite{moriah-sedgwick-twisted} states that any knot whose union with some unknot is a tunnel number one link must be $(1,1)$. That is, it has a 1-bridge presentation with respect to an unknotted torus.  However this is a contradiction, as Morimoto, Sakuma and Yokota \cite{msy} demonstrated that the knot $T(7,17,6)$ is not $(1,1)$.  It follows that any genus two Heegaard splitting of $X(L)$ induces the bipartition $\{U_0,U_1 | T(7,17,6) \}$.

(3) Note that the exterior of the link $U_0 \cup U_1$ is a product, $T^2 \times [-1,1]$.  Draw the $(7,17)$ torus knot as a curve on the level surface $T^2 \times \{0\}$ in this product.   Choose two strands of the torus knot and give them 6 half twists to obtain the twisted torus knot $T(7,17,6)$. Its union with the pair of unknots is our twisted torus link $L$.

Now, note that the $(2,5)$ curve drawn on the same level torus meets the $(7,17)$ curve in a single point.   Then the product $(2,5) \times [-1,1]$ is a properly embedded annulus in the product that meets the torus knot once, and the unknots in slopes $\frac25$ and $\frac52$, respectively. Moreover, the twisting needed to construct $T(7,17,6)$ can be performed in the complement of this annulus.   Drill out the twisted torus knot.  The annulus is punctured once (with slope $\infty=\frac10$ on the knot) and becomes an essential pair of pants $P$ in the link exterior.

Let $B \subset X(L)$ be a properly embedded M\"obius band with its boundary in the knotted component and that meets $P$ in the minimal number of components.  Because both surfaces are essential, the intersection consists of a collection of arcs that are essential in both surfaces.

In fact, there is only a single arc of intersection: if there were two or more, then there would be a pair of arcs that are parallel and adjacent on $P$ and that are also parallel on $B$.   Then the union $B' = R_P \cup R_B$, where $R_P$ and $R_B$ are the rectangles the arcs bound in $P$ and $R$, respectively, is a M\"obius band (see for example \cite{rieck-heegaard-dehn}) that can be isotoped to meet $P$ in a single arc.

However, it is also impossible for $P \cap B$ to consist of a single arc:  this implies that the M\"obius band has slope $\frac n 2$ for some $n$ as it meets the meridian $\frac 1 0$ twice.  But, any $\frac n 2 $ curve also bounds a M\"obius band in the solid torus that is attached to perform the meridional ($S^3$) filling on the knotted component.   The union of the $B$ and the M\"obius band in the solid torus is a Klein bottle embedded in $S^3$, a contradiction.

\end{proof}

Finally, for \bnot~ blocks we want a manifold for which no minimal genus Heegaard surface splits its two boundary components.  Note that $X(L)$ is almost what we want; no minimal Heegaard surface splits the two unknotted boundary components. Nonetheless, there is an inconvenient third  boundary component (the knotted one).   Can we get rid of it?

There are many results that demonstrate that after a ``sufficiently large'' Dehn filling, the filled manifold inherits the qualities of the unfilled manifold. Fortunately, that is also true for Heegaard structure \cite{moriah-rubinstein, rieck-heegaard-dehn, rieck-sedgwick-1, rieck-sedgwick-2, moriah-sedgwick-heegaard-dehn} and that is precisely what we use here:

\begin{lem}[\bnot]
\label{l-twisted-link-filled}
Let $L \subset S^3$ be the link indicated in Figure \ref{f:t7176}, and let $X(L;\gamma)$ be the manifold obtained by Dehn filling the knotted component along the slope $\gamma$.  If $d_\mathcal F(\gamma,\infty)>10$, where $d_\mathcal F$ is the distance in the Farey graph, then
\begin{enumerate}
\item $X(L;\gamma)$ has Heegaard genus two,
\item every genus two Heegaard splitting of $X(L;\gamma)$ induces the trivial boundary bipartition $\{U_0,U_1 | \emptyset \}$.
    \end{enumerate}
\end{lem}

\begin{proof}

Heegaard surfaces survive Dehn fillings.  That is, after filling any slope $\gamma$, a Heegaard surface for $X(L)$ is also a Heegaard surface for $X(L;\gamma)$.  Thus the genus of $X(L;\gamma)$ is at most 2.

We now show that under the hypothesis $\D(\gamma,\infty)>10$, every genus two Heegaard splitting of $X(L;\gamma)$ is isotopic (in $X(L;\gamma)$) to a Heegaard splitting of $X(L)$. It will follow that the genus of $X(L;\gamma)$ is exactly two, and any genus two splitting induces the desired bipartition of boundary components.

We will say that a filled manifold $X(L;\alpha)$ has a {\it new} Heegaard surface if there is a Heegaard surface $\Sigma \subset X(L;\alpha)$  for the filled manifold that is {\it not} isotopic in $X(L;\alpha)$ to a Heegaard surface for $X(L)$.  Rieck and Sedgwick \cite{rieck-sedgwick-1} have shown that there are two possibilities for a new Heegaard surface $\Sigma$, depending on whether the core of the attached solid torus is isotopic into $\Sigma$ in the filled manifold.  In either case, we can find a useful derived surface $\Sigma' \subset X(L)$ by isotoping $\Sigma$ in $X(L;\alpha)$ and then drilling out the core:  if the attached core is not isotopic into $\Sigma$, then $\Sigma$ is isotopic to a ``thick level" in some thin presentation of the core, which is a knot in $X(L;\alpha)$.   After drilling out the core, we obtain a properly embedded surface $\Sigma \subset X(L)$ that meets the knotted boundary component in curves of slope $\alpha$.  If the core {\it is} isotopic into $\Sigma$, then drilling out the core and possibly compressing, we obtain a properly embedded essential surface $\Sigma' \subset X(L)$. Its genus  is at most that of $\Sigma$ and its boundary curves meet the knotted boundary component in a slope $\alpha'$, where $\D(\alpha',\alpha)=1$.

If two different filled manifolds $X(L;\alpha)$ and $X(L;\beta)$ have new Heegaard surfaces, then the pair of bounded surfaces derived above, each either essential or ``thick,'' can be isotoped to intersect essentially (\cite{gabai-thin-position}, \cite{rieck-heegaard-dehn}).   Moreover, the previous lemma shows that there is no M\"obius band in $X(L)$ with its boundary in the knotted component.  In that case Rieck showed that  the number of intersections between the slopes $\alpha$ and $\beta$ is bounded by a quadratic function, $36 g_1 g_2 + 36 g_1 + 18 g_2 + 18$, where $g_1$ and $g_2$, $g_1 \geq g_2$, are the genera of the derived surfaces (\cite{rieck-heegaard-dehn} Theorem 5.2).   (Theorem 5.2 is stated with a stronger hypothesis, that $X(L)$ is a-cylindrical, but the proof clearly states that either the bound holds or there is a M\"obius band meeting the boundary component that was filled.)

Now, we know that the manifold $X(L,\infty)$ is the product $T^2 \times [-1,1]$ and thus has a new Heegaard surface of genus 1.  (As the knotted component is not a torus knot, in this case the derived surface is a thick level with genus 1 and slope $\infty$.)

Suppose then that $X(L,\gamma)$ has a new Heegaard surface of genus at most 2.  Then the slopes of the derived surfaces intersect at most 180 times (applying the above quadratic function with $g_1 =2 \geq g_2 = 1$) and thus have distance in the Farey graph $d_\mathcal F \leq \log_2 180 + 1 < 9$.   As the derived surface in $X(L,\gamma)$ has distance 0 or 1 from $\gamma$, we have $\D(\gamma,\infty) < 10$, a contradiction.

It follows that $X(L,\gamma)$ has no new Heegaard surfaces with genus at most 2.   Then the genus of $X(L,\gamma)$ is 2. Moreover, every genus two Heegaard surface of $X(L,\gamma)$ is isotopic in $X(L,\gamma)$ to a Heegaard surface for $X(L)$, and in particular induces the boundary bipartition $\{U_0,U_1 | \emptyset \}$.  This completes the proof.

 \end{proof}

Construct the \bnot~ blocks by using Lemmas \ref{l-triang-gluing} and \ref{l-twisted-link-filled} to glue the triangulated twisted torus link exterior to a one-tetrahedron solid torus (see for example, \cite{jaco-sedgwick}) so that $\mu$, the curve bounding a meridional disk of the solid torus, and $\infty$ the meridian of the twisted torus link, satisfy $\D(\mu,\infty) > 11$.

\subsection{Proof of Proposition \ref{t-triangulation-quadratic}}
\label{s:proof_prop}

\begin{proof}
The manifold $M_Q$ is obtained by gluing a collection of blocks along pairs of torus boundary components via high distance maps.   There is exactly one block for each term (\bvar, \band, \bor, \bnot) in $Q$, plus the \bend~ block, for a total of $|Q|+1$ blocks.

As a preprocessing step, we triangulate each of the block types so that each torus boundary component has a one-vertex triangulation. For each of the three link exteriors, use the method Weeks describes in \cite{weeks-computation} and implements in his {\it SnapPea} program, to convert the link diagrams given by Figures \ref{f:trefoilChain} and \ref{f:t7176}
to ideal triangulations of the link exteriors.  Then construct a (non-ideal) triangulation by subdividing and deleting tetrahedra meeting the ideal vertex.   Use Jaco and Rubinstein's method to convert this triangulation to a 0-efficient triangulation \cite{jaco-rubinstein-0efficient}, which has the desired property  that it restricts to a one-vertex triangulation of each torus boundary component.  For each torus boundary component of each block, use normal surface theory to identify, among essential surfaces meeting the boundary component, a surface maximizing Euler characteristic.

Let $T$ be the maximal number of tetrahedra used by one of the four triangulated blocks types. Since there are $\LQ+1$ blocks, we thus require at most $T(\LQ+1)$ tetrahedra before gluing.

There is a computable constant $K$, depending only on the homeomorphism types of the blocks, so that if any set of blocks are glued with maps of distance at least $Kg$ (relative to the boundaries, then any Heegaard surface whose genus is at most $g$ is an amalgamation of splittings of the blocks. (The proof of this is given in the appendix; distance is measured between the surfaces chosen above.) As we want to guarantee that any splitting of genus at most $|Q|+2$ is an amalgamation, it is thus sufficient to glue each pair of blocks with a map of distance $K(|Q|+2)$, which by Lemma \ref{l-triang-gluing} requires $2 K (|Q|+2)$ tetrahedra per gluing. Since each of the $|Q|+1$ blocks has at most 3 boundary components, there are at most $\frac 3 2(|Q|+1)$ pairs of boundary components to glue.  We conclude that we need at most $\frac 3 2(\LQ+1) 2K(\LQ+2)$ tetrahedra to glue the blocks.

The total number of tetrahedra required to construct $M_Q$ is then the sum of those for the blocks and those for gluings,
$$ t(M_Q) \leq T( |Q|+1) + 3K(|Q|+1)(|Q|+2)$$
which is clearly quadratically bounded in \LQ.
\end{proof}

\section{Open Questions}
\label{s:questions}

We now discuss some questions that remain.  The most obvious is:
\begin{quest}
Is {\sc Heegaard Genus $\leq g$} in \NP?
\end{quest}

Next, since the 3-sphere is, by definition, the  3-manifold with genus 0, {\sc 3-Sphere Recognition} is precisely {\sc Heegaard Genus $\leq 0$}, i.e., a special case of our general problem with fixed parameter $g=0$.  Schleimer showed that {\sc 3-Sphere Recognition} is in \NP~\cite{schleimer}.  And, using Kuperberg's work \cite{kuperberg}, Zentner showed that {\sc 3-Sphere Recognition} is also in co-\NP~ if we assume that the Generalized Riemann Hypothesis  is true \cite{zentner}.   Thus, without disproving a major conjecture, we do not expect  the special case {\sc Heegaard Genus $\leq 0$} to be  \NP-hard.  Since Heegaard genus is such an important invariant, it is worth asking about the  complexity of the problem for other small fixed values of $g$, in particular  $g \leq 2$:

\begin{quest} What is the computational complexity of deciding {\sc Heegaard Genus $\leq 1$} and {\sc Heegaard Genus $\leq 2$}?
\end{quest}

Finally, note that our construction produces non-hyperbolic manifolds because the identified torus boundary components are incompressible after gluing. It seems probable that hyperbolic  examples can be constructed by gluing together  hyperbolic block  manifolds that have higher genus boundary components.   But, the resulting manifolds would most definitely be {\it Haken} (have embedded incompressible surfaces).   Do embedded essential surfaces explain \NP-hardness or,

\begin{quest}
Is {\sc Heegaard Genus $\leq g$}   \NP-hard when restricted to the class of non-Haken manifolds?
\end{quest}

\section{appendix: Sufficiently complicated amalgamations}
\label{s:appendix}

In this section we provide a proof of the following proposition, based on several well-known results.

\begin{pro}
\label{t:amalgamation}
There is a computable constant $K$, depending only on the homeomorphism types of the blocks, so that if any set of blocks are glued with maps of distance at least $Kg$ (in the sense of Theorem \ref{t:GluingTheorem} below), then any Heegaard surface whose genus is at most $g$ is an amalgamation of splittings of the blocks.
\end{pro}

\begin{proof}
Suppose $H$ is a minimal genus Heegaard splitting of $M_Q$. It follows from the results of \cite{ScharlemannThompson} that there is a DAG $\Gamma$  such that $H$ is an amalgamation of some generalized Heegaard splitting $\bigcup _{x \in \Gamma} M_x$ of $M_Q$, such that for each $x \in \Gamma$, $V_x \cap W_x$ is {\it strongly irreducible} in $M_x$, and for each $x \ne y$, $V_x \cap W_y$ is a (possibly empty) incompressible surface in $M$. In the parlance of \cite{bachman-topindex}, both kinds of surfaces are {\it topologically minimal} in $M$. Let $\mathcal H$ denote the union of all such topologically minimal surfaces.

For each  boundary component $F$ of each block used in the original construction of $M_Q$ (see Section \ref{s:constructing}), choose a maximal Euler characteristic, properly embedded, incompressible, boundary incompressible surface in that block that is incident to $F$. Let $\mathcal S$ be the collection of these chosen surfaces. (Note that the surfaces in $\mathcal S$ need not be disjoint in each block).

Let $M_-$ and $M_+$ denote blocks used in the construction of $M_Q$, such that $M_+ \cap M_- \ne \emptyset$. Let $F$ be a component of $M_+ \cap M_-$. Then $F$ can be identified with boundary components $F_-\subset \partial M_-$ and $F_+ \subset \partial M_+$. Let $\phi : F_- \to F_+$ denote the gluing map used to attach $M_-$ to $M_+$ along $F$ in the construction of $M_Q$. Let $M_{\phi}$ denote the manifold obtained from $M_-$ and $M_+$ by gluing $F_-$ to $F_+$ via the map $\phi$. Note that $M_\phi$ may be different from $M_- \cup M_+$, as the latter manifold may be obtained from $M_-$ and $M_+$ by gluing along multiple surfaces.  However, if $\mathcal F$ denotes the collection of surfaces at the interfaces between all blocks in $M_Q$, then $M_\phi$ can be identified with a component of the complement of $\mathcal F  \setminus F$.

By \cite{BachmanSchleimerSedgwick}, we can isotope each surface in $\mathcal H$ so that it meets the complementary pieces of $\mathcal F  \setminus F$ in a collection of surfaces that are topologically minimal (in particular, either incompressible or strongly irreducible). After such an isotopy, let $H'$ denote a component of the intersection of such a surface with $M_\phi$.

The first author, building on work of Tao Li \cite{li-gluings}, proved the following theorem, restated here with notation consistent with that of the present paper:

\begin{thm}
\label{t:GluingTheorem}
(cf. \cite{bachman-gluings}, Theorem 5.4.)
 Let $S_-$ and $S_+$ denote the surfaces in $\mathcal S$ chosen to meet $F_-$ and $F_+$ in $M_-$ and $M_+$. Let $K = 24(1 - 3\chi(S_-) - 3\chi(S_+))$.
If
\[d(\phi(S_- \cap F_-), S_+\cap F_+) \ge K \cdot \mbox{\rm genus}(H)\]
then $H'$ can be isotoped to be disjoint from $F$ in $M_\phi$. \footnote{The original theorem is stated so that $H'$ is a closed surface, but this assumption is never used in the proof.}
\end{thm}

Note that $H'$ is a component of $\mathcal H \cap M_\phi$. Applying this Theorem to every such component (noting that genus$(H') \le$ genus $(H)$), we conclude $\mathcal H$ can be isotoped to be disjoint from $F$ in $M_Q$. Each surface in the resulting collection is now topologically minimal in $M_Q-F$. Repeating this argument for every surface in $\mathcal F$ shows that every surface in $\mathcal H$ can be isotoped entirely into some block. It then follows from standard arguments that each surface of $\mathcal F$ can be identified with a component of $\partial M_x$, for some  $x \in \Gamma$. Thus, for each block $B$ in $M_Q$, there is a collection of vertices $\mathcal V$ of $\Gamma$ such that $B=\bigcup _{x \in \mathcal V} M_x$. Amalgamating this generalized Heegaard splitting of $B$ then produces a Heegaard splitting of $B$. Our original Heegaard surface $H$ is then an amalgamation of these Heegaard surfaces of the blocks.
\end{proof}

\bibliographystyle{alpha}
\bibliography{HeegaardNPHard}

\begin{thebibliography}{MSTW14}

\bibitem[AHT02]{agol-hass-thurston}
Ian Agol, Joel Hass, and William Thurston.
\newblock 3-manifold knot genus is {NP}-complete.
\newblock In {\em Proceedings of the {T}hirty-{F}ourth {A}nnual {ACM}
  {S}ymposium on {T}heory of {C}omputing}, pages 761--766. ACM, New York, 2002.

\bibitem[Bac10]{bachman-topindex}
David Bachman.
\newblock Topological index theory for surfaces in 3-manifolds.
\newblock {\em Geom. Topol.}, 14(1):585--609, 2010.

\bibitem[Bac13]{bachman-gluings}
David Bachman.
\newblock Stabilizing and destabilizing {H}eegaard splittings of sufficiently
  complicated 3-manifolds.
\newblock {\em Math. Ann.}, 355(2):697--728, 2013.

\bibitem[Bdd14]{burton-verdiere-demesmay}
B.~A. {Burton}, {\'E}.~C. {de Verdi{\`e}re}, and A.~{de Mesmay}.
\newblock {On the Complexity of Immersed Normal Surfaces}.
\newblock Preprint arXiv:1412.4988, December 2014.

\bibitem[BdW16]{burton-demesmay-wagner}
B.~A. {Burton}, A.~{de Mesmay}, and U.~{Wagner}.
\newblock {Finding non-orientable surfaces in 3-manifolds}.
\newblock Preprint arXiv:0901.0208, feb 2016.

\bibitem[BS13]{burton-spreer-taut}
Benjamin~A. Burton and Jonathan Spreer.
\newblock The complexity of detecting taut angle structures on triangulations.
\newblock In {\em Proceedings of the Twenty-Fourth Annual {ACM-SIAM} Symposium
  on Discrete Algorithms, {SODA} 2013, New Orleans, Louisiana, USA, January
  6-8, 2013}, pages 168--183, 2013.

\bibitem[BSS06]{BachmanSchleimerSedgwick}
David Bachman, Saul Schleimer, and Eric Sedgwick.
\newblock Sweepouts of amalgamated 3-manifolds.
\newblock {\em Algebr. Geom. Topol.}, 6:171--194 (electronic), 2006.

\bibitem[Gab87]{gabai-thin-position}
David Gabai.
\newblock Foliations and the topology of {$3$}-manifolds. {III}.
\newblock {\em J. Differential Geom.}, 26(3):479--536, 1987.

\bibitem[Hak61]{haken-unknot}
Wolfgang Haken.
\newblock Theorie der {N}ormalfl\"achen.
\newblock {\em Acta Math.}, 105:245--375, 1961.

\bibitem[HLP99]{hass-lagarias-pippenger}
Joel Hass, Jeffrey~C. Lagarias, and Nicholas Pippenger.
\newblock The computational complexity of knot and link problems.
\newblock {\em J. ACM}, 46(2):185--211, 1999.

\bibitem[Joh90]{johannson1}
Klaus Johannson.
\newblock Heegaard surfaces in {H}aken {$3$}-manifolds.
\newblock {\em Bull. Amer. Math. Soc. (N.S.)}, 23(1):91--98, 1990.

\bibitem[Joh95]{johannson2}
Klaus Johannson.
\newblock {\em Topology and combinatorics of 3-manifolds}, volume 1599 of {\em
  Lecture Notes in Mathematics}.
\newblock Springer-Verlag, Berlin, 1995.

\bibitem[JR03]{jaco-rubinstein-0efficient}
William Jaco and J.~Hyam Rubinstein.
\newblock {$0$}-efficient triangulations of 3-manifolds.
\newblock {\em J. Differential Geom.}, 65(1):61--168, 2003.

\bibitem[JR06]{jaco-rubinstein-layered}
W.~Jaco and J.~Hyam Rubinstein.
\newblock {Layered-triangulations of 3-manifolds}.
\newblock Preprint arXiv:math/0603601, March 2006.

\bibitem[JS03]{jaco-sedgwick}
William Jaco and Eric Sedgwick.
\newblock Decision problems in the space of {D}ehn fillings.
\newblock {\em Topology}, 42(4):845--906, 2003.

\bibitem[Kob99]{kobayashi-2-bridge}
Tsuyoshi Kobayashi.
\newblock Classification of unknotting tunnels for two bridge knots.
\newblock In {\em Proceedings of the {K}irbyfest ({B}erkeley, {CA}, 1998)},
  volume~2 of {\em Geom. Topol. Monogr.}, pages 259--290 (electronic). Geom.
  Topol. Publ., Coventry, 1999.

\bibitem[Kup14]{kuperberg}
Greg Kuperberg.
\newblock Knottedness is in {\sc{np}}, modulo {GRH}.
\newblock {\em Adv. Math.}, 256:493--506, 2014.

\bibitem[Lac16]{lackenby-sublink}
Marc Lackenby.
\newblock Some conditionally hard problems on links and 3-manifolds.
\newblock Preprint arXiv:1602.08427, 2016.

\bibitem[Li10]{li-gluings}
Tao Li.
\newblock Heegaard surfaces and the distance of amalgamation.
\newblock {\em Geom. Topol.}, 14(4):1871--1919, 2010.

\bibitem[Li11]{li-heegaard-genus}
Tao Li.
\newblock An algorithm to determine the {H}eegaard genus of a 3-manifold.
\newblock {\em Geom. Topol.}, 15(2):1029--1106, 2011.

\bibitem[MR97]{moriah-rubinstein}
Yoav Moriah and Hyam Rubinstein.
\newblock Heegaard structures of negatively curved {$3$}-manifolds.
\newblock {\em Comm. Anal. Geom.}, 5(3):375--412, 1997.

\bibitem[MS98]{moriah-schultens}
Yoav Moriah and Jennifer Schultens.
\newblock Irreducible {H}eegaard splittings of {S}eifert fibered spaces are
  either vertical or horizontal.
\newblock {\em Topology}, 37(5):1089--1112, 1998.

\bibitem[MS04]{moriah-sedgwick-weakly}
Yoav Moriah and Eric Sedgwick.
\newblock Closed essential surfaces and weakly reducible {H}eegaard splittings
  in manifolds with boundary.
\newblock {\em J. Knot Theory Ramifications}, 13(6):829--843, 2004.

\bibitem[MS07]{moriah-sedgwick-heegaard-dehn}
Yoav Moriah and Eric Sedgwick.
\newblock The {H}eegaard structure of {D}ehn filled manifolds.
\newblock In {\em Workshop on {H}eegaard {S}plittings}, volume~12 of {\em Geom.
  Topol. Monogr.}, pages 233--263. Geom. Topol. Publ., Coventry, 2007.

\bibitem[MS09]{moriah-sedgwick-twisted}
Yoav Moriah and Eric Sedgwick.
\newblock Heegaard splittings of twisted torus knots.
\newblock {\em Topology Appl.}, 156(5):885--896, 2009.

\bibitem[MSTW14]{matousek-sedgwick-tancer-wagner}
Ji{\v{r}}{\'{\i}} Matou{\v{s}}ek, Eric Sedgwick, Martin Tancer, and Uli Wagner.
\newblock Embeddability in the 3-sphere is decidable.
\newblock In {\em Computational Geometry ({S}o{CG}'14)}, pages 78--84. ACM, New
  York, 2014.

\bibitem[MSY96]{msy}
Kanji Morimoto, Makoto Sakuma, and Yoshiyuki Yokota.
\newblock Examples of tunnel number one knots which have the property
  ``{$1+1=3$}''.
\newblock {\em Math. Proc. Cambridge Philos. Soc.}, 119(1):113--118, 1996.

\bibitem[Rie00]{rieck-heegaard-dehn}
Yo'av Rieck.
\newblock Heegaard structures of manifolds in the {D}ehn filling space.
\newblock {\em Topology}, 39(3):619--641, 2000.

\bibitem[RS01a]{rieck-sedgwick-2}
Yo'av Rieck and Eric Sedgwick.
\newblock Finiteness results for {H}eegaard surfaces in surgered manifolds.
\newblock {\em Comm. Anal. Geom.}, 9(2):351--367, 2001.

\bibitem[RS01b]{rieck-sedgwick-1}
Yo'av Rieck and Eric Sedgwick.
\newblock Persistence of {H}eegaard structures under {D}ehn filling.
\newblock {\em Topology Appl.}, 109(1):41--53, 2001.

\bibitem[Rub95]{rubinstein}
Joachim~H. Rubinstein.
\newblock An algorithm to recognize the {$3$}-sphere.
\newblock In {\em Proceedings of the {I}nternational {C}ongress of
  {M}athematicians, {V}ol.\ 1, 2 ({Z}\"urich, 1994)}, pages 601--611.
  Birkh\"auser, Basel, 1995.

\bibitem[Sch93]{schultens-Heegaard-product}
Jennifer Schultens.
\newblock The classification of {H}eegaard splittings for (compact orientable
  surface){$\,\times\, S\sp 1$}.
\newblock {\em Proc. London Math. Soc. (3)}, 67(2):425--448, 1993.

\bibitem[Sch11]{schleimer}
Saul Schleimer.
\newblock Sphere recognition lies in {NP}.
\newblock In {\em Low-dimensional and symplectic topology}, volume~82 of {\em
  Proc. Sympos. Pure Math.}, pages 183--213. Amer. Math. Soc., Providence, RI,
  2011.

\bibitem[ST93]{Scharlemann-Thompson}
Martin Scharlemann and Abigail Thompson.
\newblock Heegaard splittings of (surface){$\,\times\, {I}$} are standard.
\newblock {\em Math. Ann.}, 295:549--564, 1993.

\bibitem[ST94]{ScharlemannThompson}
Martin Scharlemann and Abigail Thompson.
\newblock Thin position for {$3$}-manifolds.
\newblock In {\em Geometric topology ({H}aifa, 1992)}, volume 164 of {\em
  Contemp. Math.}, pages 231--238. Amer. Math. Soc., Providence, RI, 1994.

\bibitem[SW07]{Schultens-Weidmann}
Jennifer Schultens and Richard Weidmann.
\newblock Destabilizing amalgamated {H}eegaard splittings.
\newblock {\em Geometry and Topology Monographs}, 12:319--334, 2007.

\bibitem[Tho94]{thompson}
Abigail Thompson.
\newblock Thin position and the recognition problem for {$S\sp 3$}.
\newblock {\em Math. Res. Lett.}, 1(5):613--630, 1994.

\bibitem[Wee05]{weeks-computation}
Jeff Weeks.
\newblock Computation of hyperbolic structures in knot theory.
\newblock In {\em Handbook of knot theory}, pages 461--480. Elsevier B. V.,
  Amsterdam, 2005.

\bibitem[{Zen}16]{zentner}
R.~{Zentner}.
\newblock {Integer homology 3-spheres admit irreducible representations in
  SL(2,C)}.
\newblock Preprint arXiv:1605.08530, May 2016.

\end{thebibliography}

\end{document}